\documentclass[12pt]{amsart}

\usepackage{amsfonts,amsmath,amsxtra,amsthm,amssymb,tikz,pgfplots}
\usepackage{mathrsfs}
\usepackage{times}
\usepackage{anysize}
\usepackage{graphicx,underscore}

\input xy
\xyoption{all}
\usepackage[all,knot]{xy}

\marginsize{1in}{1in}{1in}{1in}

\numberwithin{equation}{subsection}

\theoremstyle{plain}
\newtheorem{theorem}[equation]{Theorem}
\newtheorem{lemma}[equation]{Lemma}
\newtheorem{proposition}[equation]{Proposition}
\newtheorem{corollary}[equation]{Corollary}

\newtheorem{problem}[equation]{Problem}

\theoremstyle{definition}
\newtheorem{definition}[equation]{Definition}
\newtheorem{example}[equation]{Example}
\newtheorem{remark}[equation]{Remark}

\DeclareMathOperator{\Imm}{Im}

\DeclareMathOperator{\Ord}{ord}
\DeclareMathOperator{\HF}{HF}

\DeclareMathOperator{\lk}{lk}
\DeclareMathOperator{\Supp}{Supp}

\newcommand{\HFminus}{\HF^{-}}
\newcommand{\F}{\mathbb{F}}

\newcommand{\Cc}{\mathcal{C}}

\newcommand{\Vv}{\mathcal{V}}

\newcommand{\BC}{\mathbb{C}}

\newcommand{\BZ}{\mathbb{Z}}

\newcommand{\Kbar}{\overline{K}}
\newcommand{\ls}{\mathbf{LS}}
\newcommand{\BQ}{\mathbb{Q}}
\newcommand{\BH}{\mathbb{H}}
\newcommand{\lkv}{\ell}
\newcommand{\bt}{{\bf t}}

\author{Eugene Gorsky}
\thanks{The first author was partially supported by RFBR grants 13-01-00755, 16-01-00409 and NSF grants DMS-1403560, DMS-1559338.}
\address{Department of Mathematics, UC Davis, One Shields Avenue, Davis, CA 95616 }
\address{National Research University Higher School of Economics,
 Usacheva 6, Moscow, Russia}
\email{egorskiy@math.ucdavis.edu}

\author{Andr\'as N\'emethi}
\thanks{The second author was partially supported by NKFIH Grant  112735 and
ERC Adv. Grant LDTBud of A. Stipsicz at R\'enyi Institute of Math., Budapest}
\address{ Alfr\'ed R\'enyi Institute of Mathematics,
Hungarian Academy of Sciences,
Re\'altanoda utca 13-15, H-1053, Budapest, Hungary\newline
 \hspace*{4mm} ELTE - University of Budapest, Hungary \newline \hspace*{4mm}
BCAM - Basque Center for Applied Mathematics, Bilbao, Spain}
\email{nemethi.andras@renyi.mta.hu }

\title{On the set of L--space surgeries for links}

\begin{document}

\begin{abstract}
It it known that the set of L--space surgeries on a nontrivial  L--space knot is always bounded from below.
However, already for two-component torus links the set of L--space surgeries might be  unbounded from below.
For algebraic two--component links we  provide three complete characterizations  for the  boundedness from below:
one in terms of the $h$--function, one in terms of the Alexander polynomial, and one in terms of the
embedded resolution graph. They show
 that the set of L--space surgeries is bounded from below for most algebraic links.
In fact, the used property of the $h$--function is a sufficient condition for non--algebraic
L--space links as well. 
\end{abstract}

\maketitle

\section{Introduction}

\subsection{}

A 3-manifold is called an L--space, if its Heegaard-Floer homology has the minimal possible rank. L--spaces have been recently explored  and applied to various problems in low-dimensional topology \cite{OSsurg2}.
Being an L--space reflects several deep surgery, topological  and geometrical properties.
A link in $S^3$ is called an L--space link
if all sufficiently large surgeries along its components are L--spaces.

\begin{definition}
Let $L=L_1\cup\ldots\cup L_r\subset S^3$ be a link with $r$ components. We define $\ls(L)\subset \BZ^r$ to be the set
of all $r$--tuples $(d_1,\ldots,d_r)$ such that the surgery  $S^3_{d_1,\ldots,d_r}(L)$ of $S^3$ along $L$ with  coefficients $(d_1,\ldots,d_r)$ is an L--space.
\end{definition}

By definition, $L$ is an L--space link if and only if $(\BZ_{\ge N})^r\subset \ls(L)$ for some $N$.
The structure of the set $\ls$ for knots is described by the following result.

\begin{theorem}(\cite{OSsurg1,OSsurg2}, \cite[Lemma 2.13]{Hedden})
\label{th: ls knots}
Let $K$ be a nontrivial L--space knot. Then $S^3_d(K)$ is an L--space if and only if $d\ge 2g(K)-1$.
In other words, $\ls(K)=[2g(K)-1,+\infty)$.
\end{theorem}

On the other hand, already for two-component links the structure of the set $\ls$ becomes very complicated.
For example, the sets $\ls(T(2p,2q))$  for two-component torus links
were studied for $p=1$ in \cite{Liu} and for $p>1$ in \cite{GN1}, and happen to be unbounded from below (see Figure \ref{fig: t46} for
the structure of $\ls$ for the $(4,6)$ torus link). In this paper, we study the following basic question about L--space links.

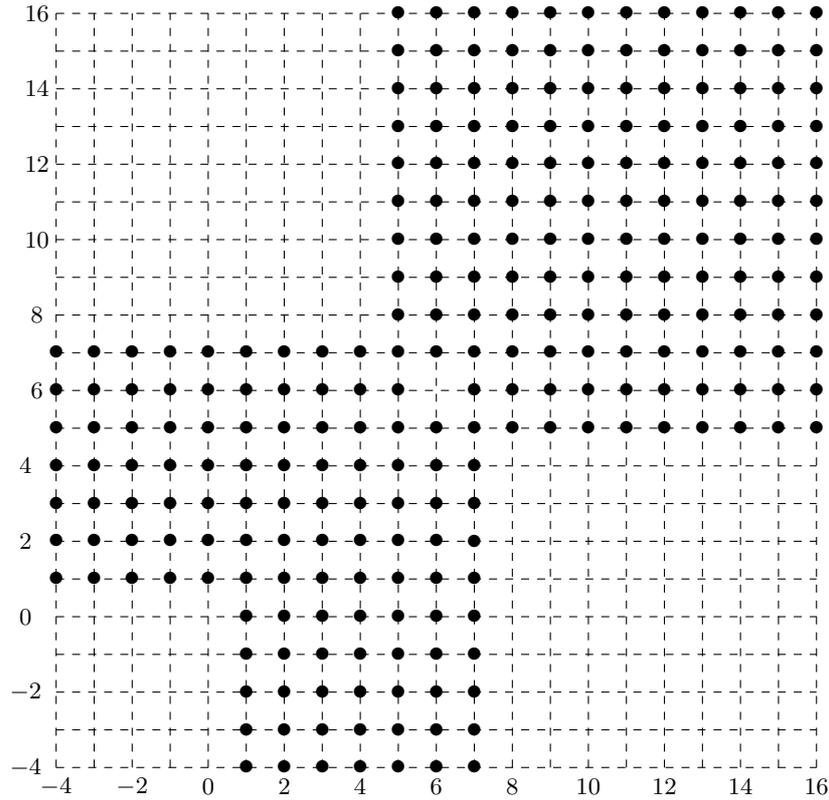
\begin{figure}
\begin{tikzpicture}[scale=0.5]
\draw [dashed] (-10,-10)--(10,-10);
\draw [dashed] (-10,-9)--(10,-9);
\draw [dashed] (-10,-8)--(10,-8);
\draw [dashed] (-10,-7)--(10,-7);
\draw [dashed] (-10,-6)--(10,-6);
\draw [dashed] (-10,-5)--(10,-5);
\draw [dashed] (-10,-4)--(10,-4);
\draw [dashed] (-10,-3)--(10,-3);
\draw [dashed] (-10,-2)--(10,-2);
\draw [dashed] (-10,-1)--(10,-1);
\draw [dashed] (-10,0)--(10,0);
\draw [dashed] (-10,1)--(10,1);
\draw [dashed] (-10,2)--(10,2);
\draw [dashed] (-10,3)--(10,3);
\draw [dashed] (-10,4)--(10,4);
\draw [dashed] (-10,5)--(10,5);
\draw [dashed] (-10,6)--(10,6);
\draw [dashed] (-10,7)--(10,7);
\draw [dashed] (-10,8)--(10,8);
\draw [dashed] (-10,9)--(10,9);
\draw [dashed] (-10,10)--(10,10);

\draw [dashed] (-10,-10)--(-10,10);
\draw [dashed] (-9,-10)--(-9,10);
\draw [dashed] (-8,-10)--(-8,10);
\draw [dashed] (-7,-10)--(-7,10);
\draw [dashed] (-6,-10)--(-6,10);
\draw [dashed] (-5,-10)--(-5,10);
\draw [dashed] (-4,-10)--(-4,10);
\draw [dashed] (-3,-10)--(-3,10);
\draw [dashed] (-2,-10)--(-2,10);
\draw [dashed] (-1,-10)--(-1,10);
\draw [dashed] (0,-10)--(0,10);
\draw [dashed] (1,-10)--(1,10);
\draw [dashed] (2,-10)--(2,10);
\draw [dashed] (3,-10)--(3,10);
\draw [dashed] (4,-10)--(4,10);
\draw [dashed] (5,-10)--(5,10);
\draw [dashed] (6,-10)--(6,10);
\draw [dashed] (7,-10)--(7,10);
\draw [dashed] (8,-10)--(8,10);
\draw [dashed] (9,-10)--(9,10);
\draw [dashed] (10,-10)--(10,10);

\draw (-10,-10.5) node {\scriptsize{$-4$}};
\draw (-8,-10.5) node {\scriptsize{$-2$}};
\draw (-6,-10.5) node {\scriptsize{$0$}};
\draw (-4,-10.5) node {\scriptsize{$2$}};
\draw (-2,-10.5) node {\scriptsize{$4$}};
\draw (0,-10.5) node {\scriptsize{$6$}};
\draw (2,-10.5) node {\scriptsize{$8$}};
\draw (4,-10.5) node {\scriptsize{$10$}};
\draw (6,-10.5) node {\scriptsize{$12$}};
\draw (8,-10.5) node {\scriptsize{$14$}};
\draw (10,-10.5) node {\scriptsize{$16$}};

\draw (-10.8,-10) node {\scriptsize{$-4$}};
\draw (-10.8,-8) node {\scriptsize{$-2$}};
\draw (-10.8,-6) node {\scriptsize{$0$}};
\draw (-10.8,-4) node {\scriptsize{$2$}};
\draw (-10.8,-2) node {\scriptsize{$4$}};
\draw (-10.5,0) node {\scriptsize{$6$}};
\draw (-10.5,2) node {\scriptsize{$8$}};
\draw (-10.5,4) node {\scriptsize{$10$}};
\draw (-10.5,6) node {\scriptsize{$12$}};
\draw (-10.5,8) node {\scriptsize{$14$}};
\draw (-10.5,10) node {\scriptsize{$16$}};

\draw (-10,-5) node {$\bullet$};
\draw (-10,-4) node {$\bullet$};
\draw (-10,-3) node {$\bullet$};
\draw (-10,-2) node {$\bullet$};
\draw (-10,-1) node {$\bullet$};
\draw (-10,0) node {$\bullet$};
\draw (-10,1) node {$\bullet$};
\draw (-9,-5) node {$\bullet$};
\draw (-9,-4) node {$\bullet$};
\draw (-9,-3) node {$\bullet$};
\draw (-9,-2) node {$\bullet$};
\draw (-9,-1) node {$\bullet$};
\draw (-9,0) node {$\bullet$};
\draw (-9,1) node {$\bullet$};
\draw (-8,-5) node {$\bullet$};
\draw (-8,-4) node {$\bullet$};
\draw (-8,-3) node {$\bullet$};
\draw (-8,-2) node {$\bullet$};
\draw (-8,-1) node {$\bullet$};
\draw (-8,0) node {$\bullet$};
\draw (-8,1) node {$\bullet$};
\draw (-7,-5) node {$\bullet$};
\draw (-7,-4) node {$\bullet$};
\draw (-7,-3) node {$\bullet$};
\draw (-7,-2) node {$\bullet$};
\draw (-7,-1) node {$\bullet$};
\draw (-7,0) node {$\bullet$};
\draw (-7,1) node {$\bullet$};
\draw (-6,-5) node {$\bullet$};
\draw (-6,-4) node {$\bullet$};
\draw (-6,-3) node {$\bullet$};
\draw (-6,-2) node {$\bullet$};
\draw (-6,-1) node {$\bullet$};
\draw (-6,0) node {$\bullet$};
\draw (-6,1) node {$\bullet$};
\draw (-5,-10) node {$\bullet$};
\draw (-5,-9) node {$\bullet$};
\draw (-5,-8) node {$\bullet$};
\draw (-5,-7) node {$\bullet$};
\draw (-5,-6) node {$\bullet$};
\draw (-5,-5) node {$\bullet$};
\draw (-5,-4) node {$\bullet$};
\draw (-5,-3) node {$\bullet$};
\draw (-5,-2) node {$\bullet$};
\draw (-5,-1) node {$\bullet$};
\draw (-5,0) node {$\bullet$};
\draw (-5,1) node {$\bullet$};
\draw (-4,-10) node {$\bullet$};
\draw (-4,-9) node {$\bullet$};
\draw (-4,-8) node {$\bullet$};
\draw (-4,-7) node {$\bullet$};
\draw (-4,-6) node {$\bullet$};
\draw (-4,-5) node {$\bullet$};
\draw (-4,-4) node {$\bullet$};
\draw (-4,-3) node {$\bullet$};
\draw (-4,-2) node {$\bullet$};
\draw (-4,-1) node {$\bullet$};
\draw (-4,0) node {$\bullet$};
\draw (-4,1) node {$\bullet$};
\draw (-3,-10) node {$\bullet$};
\draw (-3,-9) node {$\bullet$};
\draw (-3,-8) node {$\bullet$};
\draw (-3,-7) node {$\bullet$};
\draw (-3,-6) node {$\bullet$};
\draw (-3,-5) node {$\bullet$};
\draw (-3,-4) node {$\bullet$};
\draw (-3,-3) node {$\bullet$};
\draw (-3,-2) node {$\bullet$};
\draw (-3,-1) node {$\bullet$};
\draw (-3,0) node {$\bullet$};
\draw (-3,1) node {$\bullet$};
\draw (-2,-10) node {$\bullet$};
\draw (-2,-9) node {$\bullet$};
\draw (-2,-8) node {$\bullet$};
\draw (-2,-7) node {$\bullet$};
\draw (-2,-6) node {$\bullet$};
\draw (-2,-5) node {$\bullet$};
\draw (-2,-4) node {$\bullet$};
\draw (-2,-3) node {$\bullet$};
\draw (-2,-2) node {$\bullet$};
\draw (-2,-1) node {$\bullet$};
\draw (-2,0) node {$\bullet$};
\draw (-2,1) node {$\bullet$};
\draw (-1,-10) node {$\bullet$};
\draw (-1,-9) node {$\bullet$};
\draw (-1,-8) node {$\bullet$};
\draw (-1,-7) node {$\bullet$};
\draw (-1,-6) node {$\bullet$};
\draw (-1,-5) node {$\bullet$};
\draw (-1,-4) node {$\bullet$};
\draw (-1,-3) node {$\bullet$};
\draw (-1,-2) node {$\bullet$};
\draw (-1,-1) node {$\bullet$};
\draw (-1,0) node {$\bullet$};
\draw (-1,1) node {$\bullet$};
\draw (-1,2) node {$\bullet$};
\draw (-1,3) node {$\bullet$};
\draw (-1,4) node {$\bullet$};
\draw (-1,5) node {$\bullet$};
\draw (-1,6) node {$\bullet$};
\draw (-1,7) node {$\bullet$};
\draw (-1,8) node {$\bullet$};
\draw (-1,9) node {$\bullet$};
\draw (-1,10) node {$\bullet$};
\draw (0,-10) node {$\bullet$};
\draw (0,-9) node {$\bullet$};
\draw (0,-8) node {$\bullet$};
\draw (0,-7) node {$\bullet$};
\draw (0,-6) node {$\bullet$};
\draw (0,-5) node {$\bullet$};
\draw (0,-4) node {$\bullet$};
\draw (0,-3) node {$\bullet$};
\draw (0,-2) node {$\bullet$};
\draw (0,-1) node {$\bullet$};
\draw (0,1) node {$\bullet$};
\draw (0,2) node {$\bullet$};
\draw (0,3) node {$\bullet$};
\draw (0,4) node {$\bullet$};
\draw (0,5) node {$\bullet$};
\draw (0,6) node {$\bullet$};
\draw (0,7) node {$\bullet$};
\draw (0,8) node {$\bullet$};
\draw (0,9) node {$\bullet$};
\draw (0,10) node {$\bullet$};
\draw (1,-10) node {$\bullet$};
\draw (1,-9) node {$\bullet$};
\draw (1,-8) node {$\bullet$};
\draw (1,-7) node {$\bullet$};
\draw (1,-6) node {$\bullet$};
\draw (1,-5) node {$\bullet$};
\draw (1,-4) node {$\bullet$};
\draw (1,-3) node {$\bullet$};
\draw (1,-2) node {$\bullet$};
\draw (1,-1) node {$\bullet$};
\draw (1,0) node {$\bullet$};
\draw (1,1) node {$\bullet$};
\draw (1,2) node {$\bullet$};
\draw (1,3) node {$\bullet$};
\draw (1,4) node {$\bullet$};
\draw (1,5) node {$\bullet$};
\draw (1,6) node {$\bullet$};
\draw (1,7) node {$\bullet$};
\draw (1,8) node {$\bullet$};
\draw (1,9) node {$\bullet$};
\draw (1,10) node {$\bullet$};
\draw (2,-1) node {$\bullet$};
\draw (2,0) node {$\bullet$};
\draw (2,1) node {$\bullet$};
\draw (2,2) node {$\bullet$};
\draw (2,3) node {$\bullet$};
\draw (2,4) node {$\bullet$};
\draw (2,5) node {$\bullet$};
\draw (2,6) node {$\bullet$};
\draw (2,7) node {$\bullet$};
\draw (2,8) node {$\bullet$};
\draw (2,9) node {$\bullet$};
\draw (2,10) node {$\bullet$};
\draw (3,-1) node {$\bullet$};
\draw (3,0) node {$\bullet$};
\draw (3,1) node {$\bullet$};
\draw (3,2) node {$\bullet$};
\draw (3,3) node {$\bullet$};
\draw (3,4) node {$\bullet$};
\draw (3,5) node {$\bullet$};
\draw (3,6) node {$\bullet$};
\draw (3,7) node {$\bullet$};
\draw (3,8) node {$\bullet$};
\draw (3,9) node {$\bullet$};
\draw (3,10) node {$\bullet$};
\draw (4,-1) node {$\bullet$};
\draw (4,0) node {$\bullet$};
\draw (4,1) node {$\bullet$};
\draw (4,2) node {$\bullet$};
\draw (4,3) node {$\bullet$};
\draw (4,4) node {$\bullet$};
\draw (4,5) node {$\bullet$};
\draw (4,6) node {$\bullet$};
\draw (4,7) node {$\bullet$};
\draw (4,8) node {$\bullet$};
\draw (4,9) node {$\bullet$};
\draw (4,10) node {$\bullet$};
\draw (5,-1) node {$\bullet$};
\draw (5,0) node {$\bullet$};
\draw (5,1) node {$\bullet$};
\draw (5,2) node {$\bullet$};
\draw (5,3) node {$\bullet$};
\draw (5,4) node {$\bullet$};
\draw (5,5) node {$\bullet$};
\draw (5,6) node {$\bullet$};
\draw (5,7) node {$\bullet$};
\draw (5,8) node {$\bullet$};
\draw (5,9) node {$\bullet$};
\draw (5,10) node {$\bullet$};
\draw (6,-1) node {$\bullet$};
\draw (6,0) node {$\bullet$};
\draw (6,1) node {$\bullet$};
\draw (6,2) node {$\bullet$};
\draw (6,3) node {$\bullet$};
\draw (6,4) node {$\bullet$};
\draw (6,5) node {$\bullet$};
\draw (6,6) node {$\bullet$};
\draw (6,7) node {$\bullet$};
\draw (6,8) node {$\bullet$};
\draw (6,9) node {$\bullet$};
\draw (6,10) node {$\bullet$};
\draw (7,-1) node {$\bullet$};
\draw (7,0) node {$\bullet$};
\draw (7,1) node {$\bullet$};
\draw (7,2) node {$\bullet$};
\draw (7,3) node {$\bullet$};
\draw (7,4) node {$\bullet$};
\draw (7,5) node {$\bullet$};
\draw (7,6) node {$\bullet$};
\draw (7,7) node {$\bullet$};
\draw (7,8) node {$\bullet$};
\draw (7,9) node {$\bullet$};
\draw (7,10) node {$\bullet$};
\draw (8,-1) node {$\bullet$};
\draw (8,0) node {$\bullet$};
\draw (8,1) node {$\bullet$};
\draw (8,2) node {$\bullet$};
\draw (8,3) node {$\bullet$};
\draw (8,4) node {$\bullet$};
\draw (8,5) node {$\bullet$};
\draw (8,6) node {$\bullet$};
\draw (8,7) node {$\bullet$};
\draw (8,8) node {$\bullet$};
\draw (8,9) node {$\bullet$};
\draw (8,10) node {$\bullet$};
\draw (9,-1) node {$\bullet$};
\draw (9,0) node {$\bullet$};
\draw (9,1) node {$\bullet$};
\draw (9,2) node {$\bullet$};
\draw (9,3) node {$\bullet$};
\draw (9,4) node {$\bullet$};
\draw (9,5) node {$\bullet$};
\draw (9,6) node {$\bullet$};
\draw (9,7) node {$\bullet$};
\draw (9,8) node {$\bullet$};
\draw (9,9) node {$\bullet$};
\draw (9,10) node {$\bullet$};
\draw (10,-1) node {$\bullet$};
\draw (10,0) node {$\bullet$};
\draw (10,1) node {$\bullet$};
\draw (10,2) node {$\bullet$};
\draw (10,3) node {$\bullet$};
\draw (10,4) node {$\bullet$};
\draw (10,5) node {$\bullet$};
\draw (10,6) node {$\bullet$};
\draw (10,7) node {$\bullet$};
\draw (10,8) node {$\bullet$};
\draw (10,9) node {$\bullet$};
\draw (10,10) node {$\bullet$};

\end{tikzpicture}
\caption{The set $\ls$ for the $(4,6)$ torus link}
\label{fig: t46}
\end{figure}

\begin{problem}
\label{problem}
For which L--space links the set $\ls(L)$ is bounded from below?
\end{problem}

Note that by a theorem of Liu \cite{Liu} the Heegaard--Floer homology of any surgery on a 2-component L--space link
is completely determined by its Heegaard-Floer link homology, which, in its turn, is determined by the bivariate Alexander polynomial.
However, it appears to be hard to use this algorithm directly to determine the set $\ls(L)$. We give the following partial answer.

Assume that $L$ has 2 components.
Let  $h$ be the $h$--function for $L$ (defined in \cite{GN2}),
$h_i$ are the $h$--functions for $L_i$ and $v^*$ is the point naturally dual to $v$,
see Definition \ref{very good} for all details.
A point $v=(v_1,v_2)\in \BZ^2$ is called {\em good} for $L$, if $h(v_1,v_2)>h_1(v_1)$ and $h(v_1,v_2)>h_2(v_2)$. It is called {\em very good}, if both $v$ and $v^*$ are good.

\begin{theorem}
\label{th:bounded}
Suppose that for a 2-component L--space link $L$ there is a very good point $v\in \BZ^2$.
Then $\ls(L)$ is bounded from below, moreover,
$$
\ls(L)\subset \{(d_1,d_2):d_1> 0, d_2> 0, d_1d_2 > l^2\},
$$
where $l$ is the linking number between $L_1$ and $L_2$.
\end{theorem}

The proof uses Heegaard Floer link homology, especially
properties of the surgery complex developed in \cite{MO,Liu}.

Informally, Theorem \ref{th:bounded} shows that `for most' L--space links the set $\ls(L)$ is bounded from below.  For algebraic links  we will provide several
characterizations of the boundedness property.
The simplest case with $\ls(L)$ bounded from below is provided by
the link of singularity $\{(x^2-y^3)(x^3-y^2)=0\}$, consisting of two trefoils with linking number 4.
See Figure \ref{fig: ls two cusps} for the shape of $\ls(L)$.
However, the above Theorem  can  also be used for non-algebraic links:
see Example \ref{whitehead} for the Whitehead link, where  the set $\ls$ was already described in \cite{Liu}.

\begin{figure}
\begin{tikzpicture}[scale=0.5]
\draw [dashed] (-10,-10)--(10,-10);
\draw [dashed] (-10,-9)--(10,-9);
\draw [dashed] (-10,-8)--(10,-8);
\draw [dashed] (-10,-7)--(10,-7);
\draw [dashed] (-10,-6)--(10,-6);
\draw [dashed] (-10,-5)--(10,-5);
\draw [dashed] (-10,-4)--(10,-4);
\draw [dashed] (-10,-3)--(10,-3);
\draw [dashed] (-10,-2)--(10,-2);
\draw [dashed] (-10,-1)--(10,-1);
\draw [dashed] (-10,0)--(10,0);
\draw [dashed] (-10,1)--(10,1);
\draw [dashed] (-10,2)--(10,2);
\draw [dashed] (-10,3)--(10,3);
\draw [dashed] (-10,4)--(10,4);
\draw [dashed] (-10,5)--(10,5);
\draw [dashed] (-10,6)--(10,6);
\draw [dashed] (-10,7)--(10,7);
\draw [dashed] (-10,8)--(10,8);
\draw [dashed] (-10,9)--(10,9);
\draw [dashed] (-10,10)--(10,10);

\draw [dashed] (-10,-10)--(-10,10);
\draw [dashed] (-9,-10)--(-9,10);
\draw [dashed] (-8,-10)--(-8,10);
\draw [dashed] (-7,-10)--(-7,10);
\draw [dashed] (-6,-10)--(-6,10);
\draw [dashed] (-5,-10)--(-5,10);
\draw [dashed] (-4,-10)--(-4,10);
\draw [dashed] (-3,-10)--(-3,10);
\draw [dashed] (-2,-10)--(-2,10);
\draw [dashed] (-1,-10)--(-1,10);
\draw [dashed] (0,-10)--(0,10);
\draw [dashed] (1,-10)--(1,10);
\draw [dashed] (2,-10)--(2,10);
\draw [dashed] (3,-10)--(3,10);
\draw [dashed] (4,-10)--(4,10);
\draw [dashed] (5,-10)--(5,10);
\draw [dashed] (6,-10)--(6,10);
\draw [dashed] (7,-10)--(7,10);
\draw [dashed] (8,-10)--(8,10);
\draw [dashed] (9,-10)--(9,10);
\draw [dashed] (10,-10)--(10,10);

\draw (-10,-10.5) node {\scriptsize{$-4$}};
\draw (-8,-10.5) node {\scriptsize{$-2$}};
\draw (-6,-10.5) node {\scriptsize{$0$}};
\draw (-4,-10.5) node {\scriptsize{$2$}};
\draw (-2,-10.5) node {\scriptsize{$4$}};
\draw (0,-10.5) node {\scriptsize{$6$}};
\draw (2,-10.5) node {\scriptsize{$8$}};
\draw (4,-10.5) node {\scriptsize{$10$}};
\draw (6,-10.5) node {\scriptsize{$12$}};
\draw (8,-10.5) node {\scriptsize{$14$}};
\draw (10,-10.5) node {\scriptsize{$16$}};

\draw (-10.8,-10) node {\scriptsize{$-4$}};
\draw (-10.8,-8) node {\scriptsize{$-2$}};
\draw (-10.8,-6) node {\scriptsize{$0$}};
\draw (-10.8,-4) node {\scriptsize{$2$}};
\draw (-10.8,-2) node {\scriptsize{$4$}};
\draw (-10.5,0) node {\scriptsize{$6$}};
\draw (-10.5,2) node {\scriptsize{$8$}};
\draw (-10.5,4) node {\scriptsize{$10$}};
\draw (-10.5,6) node {\scriptsize{$12$}};
\draw (-10.5,8) node {\scriptsize{$14$}};
\draw (-10.5,10) node {\scriptsize{$16$}};

\draw (-3,0) node {$\bullet$};
\draw (-3,1) node {$\bullet$};
\draw (-3,2) node {$\bullet$};
\draw (-3,3) node {$\bullet$};
\draw (-3,4) node {$\bullet$};
\draw (-3,5) node {$\bullet$};
\draw (-3,6) node {$\bullet$};
\draw (-3,7) node {$\bullet$};
\draw (-3,8) node {$\bullet$};
\draw (-3,9) node {$\bullet$};
\draw (-3,10) node {$\bullet$};
\draw (-2,-1) node {$\bullet$};
\draw (-2,0) node {$\bullet$};
\draw (-2,1) node {$\bullet$};
\draw (-2,2) node {$\bullet$};
\draw (-2,3) node {$\bullet$};
\draw (-2,4) node {$\bullet$};
\draw (-2,5) node {$\bullet$};
\draw (-2,6) node {$\bullet$};
\draw (-2,7) node {$\bullet$};
\draw (-2,8) node {$\bullet$};
\draw (-2,9) node {$\bullet$};
\draw (-2,10) node {$\bullet$};
\draw (-1,-2) node {$\bullet$};
\draw (-1,-1) node {$\bullet$};
\draw (-1,0) node {$\bullet$};
\draw (-1,1) node {$\bullet$};
\draw (-1,2) node {$\bullet$};
\draw (-1,3) node {$\bullet$};
\draw (-1,4) node {$\bullet$};
\draw (-1,5) node {$\bullet$};
\draw (-1,6) node {$\bullet$};
\draw (-1,7) node {$\bullet$};
\draw (-1,8) node {$\bullet$};
\draw (-1,9) node {$\bullet$};
\draw (-1,10) node {$\bullet$};
\draw (0,-3) node {$\bullet$};
\draw (0,-2) node {$\bullet$};
\draw (0,-1) node {$\bullet$};
\draw (0,0) node {$\bullet$};
\draw (0,1) node {$\bullet$};
\draw (0,2) node {$\bullet$};
\draw (0,3) node {$\bullet$};
\draw (0,4) node {$\bullet$};
\draw (0,5) node {$\bullet$};
\draw (0,6) node {$\bullet$};
\draw (0,7) node {$\bullet$};
\draw (0,8) node {$\bullet$};
\draw (0,9) node {$\bullet$};
\draw (0,10) node {$\bullet$};
\draw (1,-3) node {$\bullet$};
\draw (1,-2) node {$\bullet$};
\draw (1,-1) node {$\bullet$};
\draw (1,0) node {$\bullet$};
\draw (1,1) node {$\bullet$};
\draw (1,2) node {$\bullet$};
\draw (1,3) node {$\bullet$};
\draw (1,4) node {$\bullet$};
\draw (1,5) node {$\bullet$};
\draw (1,6) node {$\bullet$};
\draw (1,7) node {$\bullet$};
\draw (1,8) node {$\bullet$};
\draw (1,9) node {$\bullet$};
\draw (1,10) node {$\bullet$};
\draw (2,-3) node {$\bullet$};
\draw (2,-2) node {$\bullet$};
\draw (2,-1) node {$\bullet$};
\draw (2,0) node {$\bullet$};
\draw (2,1) node {$\bullet$};
\draw (2,2) node {$\bullet$};
\draw (2,3) node {$\bullet$};
\draw (2,4) node {$\bullet$};
\draw (2,5) node {$\bullet$};
\draw (2,6) node {$\bullet$};
\draw (2,7) node {$\bullet$};
\draw (2,8) node {$\bullet$};
\draw (2,9) node {$\bullet$};
\draw (2,10) node {$\bullet$};
\draw (3,-3) node {$\bullet$};
\draw (3,-2) node {$\bullet$};
\draw (3,-1) node {$\bullet$};
\draw (3,0) node {$\bullet$};
\draw (3,1) node {$\bullet$};
\draw (3,2) node {$\bullet$};
\draw (3,3) node {$\bullet$};
\draw (3,4) node {$\bullet$};
\draw (3,5) node {$\bullet$};
\draw (3,6) node {$\bullet$};
\draw (3,7) node {$\bullet$};
\draw (3,8) node {$\bullet$};
\draw (3,9) node {$\bullet$};
\draw (3,10) node {$\bullet$};
\draw (4,-3) node {$\bullet$};
\draw (4,-2) node {$\bullet$};
\draw (4,-1) node {$\bullet$};
\draw (4,0) node {$\bullet$};
\draw (4,1) node {$\bullet$};
\draw (4,2) node {$\bullet$};
\draw (4,3) node {$\bullet$};
\draw (4,4) node {$\bullet$};
\draw (4,5) node {$\bullet$};
\draw (4,6) node {$\bullet$};
\draw (4,7) node {$\bullet$};
\draw (4,8) node {$\bullet$};
\draw (4,9) node {$\bullet$};
\draw (4,10) node {$\bullet$};
\draw (5,-3) node {$\bullet$};
\draw (5,-2) node {$\bullet$};
\draw (5,-1) node {$\bullet$};
\draw (5,0) node {$\bullet$};
\draw (5,1) node {$\bullet$};
\draw (5,2) node {$\bullet$};
\draw (5,3) node {$\bullet$};
\draw (5,4) node {$\bullet$};
\draw (5,5) node {$\bullet$};
\draw (5,6) node {$\bullet$};
\draw (5,7) node {$\bullet$};
\draw (5,8) node {$\bullet$};
\draw (5,9) node {$\bullet$};
\draw (5,10) node {$\bullet$};
\draw (6,-3) node {$\bullet$};
\draw (6,-2) node {$\bullet$};
\draw (6,-1) node {$\bullet$};
\draw (6,0) node {$\bullet$};
\draw (6,1) node {$\bullet$};
\draw (6,2) node {$\bullet$};
\draw (6,3) node {$\bullet$};
\draw (6,4) node {$\bullet$};
\draw (6,5) node {$\bullet$};
\draw (6,6) node {$\bullet$};
\draw (6,7) node {$\bullet$};
\draw (6,8) node {$\bullet$};
\draw (6,9) node {$\bullet$};
\draw (6,10) node {$\bullet$};
\draw (7,-3) node {$\bullet$};
\draw (7,-2) node {$\bullet$};
\draw (7,-1) node {$\bullet$};
\draw (7,0) node {$\bullet$};
\draw (7,1) node {$\bullet$};
\draw (7,2) node {$\bullet$};
\draw (7,3) node {$\bullet$};
\draw (7,4) node {$\bullet$};
\draw (7,5) node {$\bullet$};
\draw (7,6) node {$\bullet$};
\draw (7,7) node {$\bullet$};
\draw (7,8) node {$\bullet$};
\draw (7,9) node {$\bullet$};
\draw (7,10) node {$\bullet$};
\draw (8,-3) node {$\bullet$};
\draw (8,-2) node {$\bullet$};
\draw (8,-1) node {$\bullet$};
\draw (8,0) node {$\bullet$};
\draw (8,1) node {$\bullet$};
\draw (8,2) node {$\bullet$};
\draw (8,3) node {$\bullet$};
\draw (8,4) node {$\bullet$};
\draw (8,5) node {$\bullet$};
\draw (8,6) node {$\bullet$};
\draw (8,7) node {$\bullet$};
\draw (8,8) node {$\bullet$};
\draw (8,9) node {$\bullet$};
\draw (8,10) node {$\bullet$};
\draw (9,-3) node {$\bullet$};
\draw (9,-2) node {$\bullet$};
\draw (9,-1) node {$\bullet$};
\draw (9,0) node {$\bullet$};
\draw (9,1) node {$\bullet$};
\draw (9,2) node {$\bullet$};
\draw (9,3) node {$\bullet$};
\draw (9,4) node {$\bullet$};
\draw (9,5) node {$\bullet$};
\draw (9,6) node {$\bullet$};
\draw (9,7) node {$\bullet$};
\draw (9,8) node {$\bullet$};
\draw (9,9) node {$\bullet$};
\draw (9,10) node {$\bullet$};
\draw (10,-3) node {$\bullet$};
\draw (10,-2) node {$\bullet$};
\draw (10,-1) node {$\bullet$};
\draw (10,0) node {$\bullet$};
\draw (10,1) node {$\bullet$};
\draw (10,2) node {$\bullet$};
\draw (10,3) node {$\bullet$};
\draw (10,4) node {$\bullet$};
\draw (10,5) node {$\bullet$};
\draw (10,6) node {$\bullet$};
\draw (10,7) node {$\bullet$};
\draw (10,8) node {$\bullet$};
\draw (10,9) node {$\bullet$};
\draw (10,10) node {$\bullet$};

\end{tikzpicture}
\caption{The set $\ls$ for a pair of ``transversal" trefoils with linking number 4}
\label{fig: ls two cusps}
\end{figure}
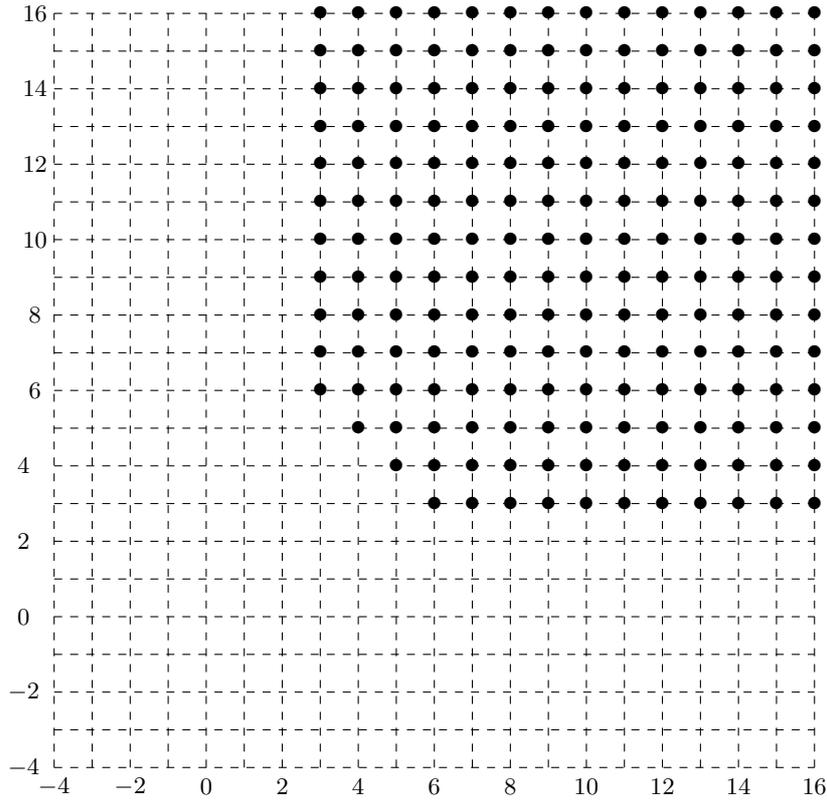

Still, there are large classes of L--space links such that $\ls(L)$  is unbounded from below.
\begin{example}
\label{th:cables}
Suppose  that $K$ is an L--space knot, $m$ and $n$ are
 positive coprime integers and $n/m>2g(K)-1$. By \cite{GHom}, the two-component cable link $K_{2m,2n}$ is an L--space link.  Then  the set $\ls(K_{2m,2n})$ is unbounded from below.
For the proof and for other examples see section \ref{s:EXAMPLES}.
\end{example}
\subsection{}

For algebraic 2--component links
the next Theorem \ref{th:alg alex}
characterizes completely all cases when $\ls(L)$ is unbounded from below.


Consider a plane curve singularity germ $C=C_1\cup C_2\subset (\BC^2,0)$  with two  components. Its intersection $L=L_1\cup L_2$ with a
small sphere centered at the origin is called an algebraic link. By \cite{GN1} all algebraic links are L--space links.

 Let $\Delta(t_1,t_2)=\sum a_{v_1,v_2}t_1^{v_1}t_2^{v_2}=\sum_{v\in \BZ^2} a_vt^v$ denote the
  Alexander polynomial
of $L=L_1\cup L_2$. It is also a complete invariant of the
embedded topological type \cite{Yam}. For its relation with other invariants
 and several properties see \cite{GN2}. The relation between $\Delta$, the   $h$--functions
  and the semigroup of the singularity is reviewed in Subsection \ref{ss:algfacts}.
 It is known  that
\begin{equation}\label{eq:av}
a_v\in \{0,1\} \ \ \mbox{for all $v$}.
\end{equation}
 Define the set
$
\Supp(\Delta)=\{v\in \BZ^2:a_v=1\}
$
and the  partial order on $\BZ^2$ by
$$
(u_1,u_2)\preceq (v_1,v_2)\ \Leftrightarrow\ u_1\le v_1\ \text{and}\ u_2\le v_2.
$$
We say that $\Delta$ is {\em of ordered type}, if for all $u,v\in \Supp(\Delta)$ one has either $u\preceq v$ or
$v\preceq u$.

Furthermore, each $L_i$ is an  iterated torus knot,  and as such, whenever it is non--trivial
there exists a unique integer $m_i$ such that
$S^3_{m_i}(L_i)$ is reducible. If $L_i$ is the unknot then  we set $m_i=1$.

\begin{theorem}
\label{th:alg alex}
For  a 2--component  algebraic link $L$  the following facts are equivalent:
\begin{enumerate}
\item $\ls(L)$ is bounded from below;
\item the intersections of $\ls(L)$ with the
lines $\{m_1\}\times \BZ$ and $\BZ\times \{m_2\}$ are both bounded from below,
\item there exists a very good point $v\in\BZ^2$ for $L$; 
\item  $\Delta(L)$ is not of ordered type.
\end{enumerate}
\end{theorem}



The proof uses several ingredients, including theory of normal surface singularities and
classification and properties of algebraic plane curve singularities.
In fact, we even add another equivalent criterion to the above list, which is formulated
in terms of the Artin's minimal cycle
\cite{Artin62,Artin66} (associated with
negative definite graph manifolds).


\subsection{}
The organization of the paper is the following.

In section \ref{s:LL1} we introduce notations and we recall basic facts regarding L--space links.

In section \ref{s:LL2} we recall the needed results regarding
 Link Floer homology and  surgery complexes  (following
\cite{Liu} and \cite{MO}) and we prove Theorem \ref{th:bounded}.

In section \ref{ss:graph mfd} we treat the combinatorics of connected negative definite graphs.
The interest in them is motivated by the fact that graph manifolds associated with such graphs are exactly the
links of normal surface singularities. For such 3--manifold, by a result of  second author \cite{Nnew},
being an  $L$--space can be reinterpreted by the `rationality' of the graph  (in the sense of
Artin \cite{Artin62,Artin66}).
We discuss properties of rational graphs, including Laufer's algorithm \cite{Laufer72}, one of the main tools
of the present note. Here a key `simplicity' property is also introduced.

We also  prove the next
 general statement of independent interest (see section \ref{ss:graph mfd} for all necessary definitions).

\begin{theorem}
Let $Y$ be a graph manifold corresponding to the negative definite rational graph $\Gamma$, and let $K_v $ be the knot in $Y$ corresponding to a vertex $v$ of $\Gamma$.
Then $Y_d(K_v)$ is an $L$ space for $d\ll 0$ if and only if the coefficient of $E_v$ in the minimal cycle of $\Gamma$ equals 1.
\end{theorem}

In section \ref{s:alglinks} we discuss invariants of algebraic links: semigroup, Alexander polynomial,
$h$--function, and several relations connecting them. We also establish certain `arithmetical'
properties of determinants of subgraphs,
 which will be crucial in the discussion of the orderability of the support of the
Alexander polynomial.

In section \ref{s:Trivi} we characterize the ($d_1\gg 0$, $d_2\ll 0$) region of $\ls(L)$ via the following results.

\begin{theorem}
\label{unknotintro}
(a)  Assume  that $L\subset S^3$ is an L--space link with two components,
 and  $S^{3}_{d_1,d_2}(L)$ is an L--space for some integers $d_1\gg 0, d_2\ll 0$.
Then $L_2$ is an unknot.


(b) Assume  that $L$ is an algebraic link with two components associated with the curve singularity
$(C,0)\subset (\BC^2,0)$. Then the following facts are equivalent.

(1)  $L_2$ is an unknot, or equivalently, $(C_2,0)$ is smooth;

(2)  $(d_1,d_2)\in \ls(L)$  for any  $d_1\gg 0$ and  $d_2\ll 0$;

(3)  $(d_1,d_2)\in \ls(L)$  for any  $d_1\geq m_1$ and  $d_2\ll 0$;

(4) if\, $\Gamma$ is the embedded resolution graph of $(C,0)\subset (\BC^2,0)$, and
$v_2$ supports the arrowhead of $L_2$ then $v_2$ is simple vertex of $\Gamma$.
\end{theorem}

Some parts of Theorem \ref{th:alg alex} follow from the constructions and results established in different sections.
Section \ref{s:PROOFS} finishes the proof. In section \ref{s:EXAMPLES}
we present several  examples illustrating the main results.

\subsection{}
Recently appeared  articles \cite{HRRW,RR,R} discuss the set of {\it rational} L--space filling slopes
$\ls_{\BQ}(L)\subset \BQ^2$
for a 3--manifold with torus boundary.  Clearly, $\ls(L)=\ls_{\BQ}(L)\cap \BZ^2$. It  follows from \cite[Theorem 1.6]{RR} that every horizontal (or vertical) section of $\ls_{\BQ}(L)$ is either empty, or it is an interval
(maybe consisting of one point or half-infinite) or it is a
complement to an interval. This result combined with our statement does not
prove the analogue of Theorem \ref{th:alg alex} for $\ls_\BQ(L)$. We will come back to this extension
(and other relations with  \cite{RR}) in a forthcoming  work.

\subsection{Acknowledgements}
The authors are grateful to Jennifer Hom, Yajing Liu, Sarah Rasmussen and Jacob Rasmussen for the useful discussions.
E. G. would like to thank R\'enyi Mathematical Institute (Budapest, Hungary)
for the hospitality, and Russian Academic Excellence Project 5-100. Many computations of Heegaard Floer homology for surgeries on algebraic links were done with the help of the program \cite{JProgram} written by Jonathan Hanselman.

\section{L--spaces and L--space links}\label{s:LL1}

\subsection{L--spaces}\label{ss:gm}
Given a 3--manifold $M$, we denote by $\HFminus(M)$  the
minus version  of its Heegaard Floer homology of $M$, cf.  \cite{OS}. It
canonically splits as a direct sum over the $spin^c$ structures of $M$:
$$
\HFminus(M)=\bigoplus_{s\in H_1(M)}\HFminus(M,s).
$$
$\HFminus(M)$ admits an action of an operator $U$ of homological degree $(-2)$,  which preserves this decomposition.
\begin{definition}
A rational homology sphere $M$ is called an L--space, if $\HFminus(M,s)$ is isomorphic as $\F[U]$-module to $\F[U]$ for all $s$.
\end{definition}
We are mostly interested in rational homology spheres, and
specifically in graph manifolds.
%
An important family of graph manifolds are given by links of complex
normal  surface singularities: they are graph manifolds associated with
connected negative definite graphs. In this way the link constitute a bridge between topological and analytical invariants. This is reflected totally in the next characterization of L--spaces given by
the second author.

\begin{theorem}[\cite{Nnew}]\label{th:ratL}
A  graph manifold associated with a connected and  negative definite
plumbing graph is an L--space if and only if the graph is rational.
\end{theorem}

Rational graphs are described  in a  purely combinatorial way,
for more details see \cite{Artin62,Artin66, Laufer72} and section \ref{ss:graph mfd} here.
Since they are stable by taking subgraphs or decreasing the Euler decorations
of the graph (see \cite{Laufer72}),  one has
 the following.

\begin{corollary}[\cite{Nnew}]\label{cor:rat}
Suppose that a negative definite graph $\Gamma$ defines an L--space (e.g. it represents $S^3$).
If\,  $\Gamma'$ is either a subgraph of\,  $\Gamma$, or it is obtained from $\Gamma$ by decreasing the
Euler decorations, then  $\Gamma'$ defines an L--space too.
\end{corollary}
\noindent
In this note we focus on surgery 3--manifolds $S^3_{d_1,\ldots, d_r}(L)$,
where $L=\{L_i\}_{i=1}^r$ is a link of $S^3$.
\begin{definition}
$L\subset S^3$ is called an L--space link, if the surgery  manifold $S^3_d(L)=S^3_{d_1,\ldots , d_r}(L)$ is an L--space for $d_i\gg 0$, $i=1,\ldots, r$.
\end{definition}
The basic examples we treat are the algebraic links
determined by (embedded) plane curve singularities
(however several of our results generalise for arbitrary links as well).
Algebraic plane curves are coded by their embedded resolution graphs,
which are connected negative definite graphs (representing $S^3$) endowed with
arrowhead vertices (representing the link components)
\cite{EN,Neumann}. Usually if $I$ is the intersection form of a graph $\Gamma$, then we define the determinant of $\Gamma$
 as $\det(\Gamma):=\det(-I)$. If the algebraic link is coded in the graph $\Gamma$, and the arrowhead of $L_i$ is supported by the vertex $v_i$ then we set $m_i:=
\det (\Gamma\setminus v_i)$.

\begin{theorem}[\cite{GN1}]
If $L$ is an algebraic link, and $d_i>m_i$ for all $i$,
then $S^3_d$ is an L--space. 
\end{theorem}
In fact, if the supporting vertices $v_i$ are all distinct,
then $S^3_d(L)$ is an L--space whenever $d_i\geq m_i$ for all $i$, cf. \cite{GN1}.
For algebraic links and for any $d\in \BZ^r$, the surgery manifolds $S^3_d(L)$ are graph manifolds, see e.g. \cite{NSurg,NGr}.
The construction of these graphs  runs as follows. Given a plane curve singularity $C$,
consider its (not necessarily minimal)
embedded  good resolution obtained by a sequence of blowups.
Let $\Gamma$ be the  dual graph and $\{v_i\}_i$ the supporting vertices of the
arrowheads representing $\{L_i\}_i$ as above.
Then we obtain the graph of  $S^3_d(L)$  from $\Gamma$
if we replace each arrowhead representing $L_i$  by a genuine vertex (connected by an edge to
$v_i$) and endow it with self-intersection $d_i-m_i$.
(We supply here another interpretation of the integers
$m_i$: if $C_i$ is the curve component providing $L_i$,
and $v_i$ is the vertex representing the irreducible exceptional curve $E_i$, then $m_i$ is the  multiplicity along $E_i$ of the total transform of $C_i$.)

\subsection{Notations} Regarding links and their surgeries we adopt the following notations.

Define a partial order on $\BZ^r$ by
$
u\preceq v \ \text{if}\ u_i\le v_i\ \text{for all}\ i.
$
For $u,v\in \BZ^r$ set
$$
\inf(u,v):=(\min(u_1,v_1),\ldots,\min(u_r,v_r)),\ \sup(u,v):=(\max(u_1,v_1),\ldots,\max(u_r,v_r)).
$$

 If $L$ is a link with $r$ components then define   $L_K$ as
  the sub--link whose components are indexed
by the  subset $K\subset \{1,\ldots,r\}$.
Let $l_{ij}$ denote the linking number between the components $L_i$ and $L_j$ ($i\not=j$).
Following \cite{Kirby},
to a vector $(d_1,\ldots,d_r)$ of surgery coefficients
we associate the {\em framing matrix} $\Lambda=\Lambda(d)$ with entries
$$
\Lambda_{ij}=\begin{cases}
d_i\ \ \ \text{if}\ i=j\\
l_{ij}\ \ \ \text{if}\ i\neq j.\\
\end{cases}
$$
We will denote the $i$-th row of $\Lambda$ by $\Lambda_i$, and for $K\subset \{1,\ldots,r\}$ define
$
\Lambda_K:=\sum_{i\in K}\Lambda_i.
$
E.g., for $r=2$, we get (with $l=l_{12}$)
$$
\Lambda=\left(\begin{matrix}
d_1 & l\\
l & d_2\\
\end{matrix}
\right).
$$
If $S^3_d(L)$ is a rational homology sphere then the order of its first homology is $|\det(\Lambda)|$.

We define the vector $c(L)=(c_1,\ldots,c_r)$ by $c_i=2g(L_i)+\sum_{j\neq i}l_{ij}$.
Given $v\in \BZ^r$, we set
  $$v^*:=c(L)-v.$$
  For $K\subset \{1,\ldots,r\}$, we define $v_K$ as the projection of $v$ to the coordinate subspace labeled by $K$.
Finally,  $\Kbar:=\{1,\ldots,r\}\setminus K$. We work over the field $\F=\BZ/2\BZ$.

\section{Link Floer homology and surgeries on L--space links}\label{s:LL2}

In this section we describe the multi-component version  of the surgery complex, following \cite{Liu} and \cite{MO}.
We assume that $L$ is an L--space link, then by \cite[Lemma 1.10]{Liu}  all its sublinks $L_K$ are L--space links too.

\subsection{Link Floer homology}\label{ss:LFH}
An $r$--component link $L$ in $S^3$ defines a $\BZ^r$ filtration (called the Alexander filtration) on the Heegaard Floer complex for $S^3$ \cite{OSlink}. This filtration is usually labeled by the lattice
$$
\BH(L):=\BZ^r+\lkv,\ \text{where}\  \lkv:=(l_1,\ldots,l_r),\ l_i=
(\textstyle{\sum_{j\neq i}}l_{ij})/2.
$$
For every sublink $L_K\subset L$ there is a natural projection map
$$
\pi_K:\BH(L)\to \BH(L_K),\ \pi_k(v)=(v-\lkv(L))_{K}+\lkv(L_K).
$$
However, by technical reasons (to match with the Hilbert function of algebraic links and with the notations of \cite{GN2}) we prefer to work with the lattice $\BZ^r$ instead of $\BH(L)$ and reverse the direction of the Alexander filtration. This is done via the map of lattices $\phi_L:\BH(L)\to \BZ^r$
\begin{equation}
\label{eqn: phi}
v\mapsto \phi_{L}(v):=-v+c(L)/2.
\end{equation}
Then in the following diagram of projections commute: 

\begin{center}
\begin{tabular}{ccc}
$\BH(L)$ & $\stackrel{\phi_{L}}{\longrightarrow}$ & $\BZ^r$\\
$\pi_K\downarrow$ & & $\downarrow \small{v \mapsto v_K}$\\
$\BH(L_K)$& $\stackrel{\phi_{L_K}}{\longrightarrow}$ & $\BZ^{|K|}$\\
\end{tabular}
\end{center}

With these notations, we define  a subcomplex $A_{K}^{-}(v):=A_{L_K}^{-}(v_K)$ for every $v\in \BZ^r$, which depends only on
the projection $v_K$ onto the sublattice labeled by $K$.
It is spanned by the generators with $i$-th Alexander filtration
greater than or equal to $v_i$
for all $i\in K$. It is known \cite{OSlink} that
\begin{equation}
\label{stabilization}
A^{-}(v)\simeq A_{K}^{-}(v)\ \text{if}\ v_i\ll 0\ \text{for}\ i\notin K.
\end{equation}

If $L$ is an L--space link, it follows from \cite[Theorem 10.1]{MO} that
$H_{*}(A_{K}^{-}(v))\simeq \F[U][-2h_K(v)]$, where $h_K(v)=h_K(v_K)$ is a certain integer-valued function.
(This is the definition of the $h$--function.)
It is proven in \cite{GN2} that this
function is completely determined by the multi-variable Alexander polynomial of  $L_K$. It follows from
\eqref{stabilization} (or see \cite{GN2}) that
\begin{equation}
\label{stabilization h}
h(v)=h_{K}(v)\ \text{if}\ v_i\ll 0\ \text{for}\ i\notin K.
\end{equation}
The function $h$ is weakly increasing:
$$
h(v)\ge h(w)\ \text{if}\ v\succeq w,
$$
$h(v)\ge 0$ for all $v$,  and
\begin{equation}
\label{h increase}
h_{K_2}(v)\ge h_{K_1}(v)\ \text{if}\ K_1\subset K_2.
\end{equation}
We will also need the next symmetry property of
 the $h$--function (cf.  \cite[Lemma 5.5]{Liu}):
\begin{equation}
\label{symmetry}
h(v^*)=h(v)-|v|+|c(L)|/2.
\end{equation}
Note that after applying $\phi_L$, one gets $\phi_L(-v)=c(L)-\phi_L(v).$
This yields a simpler equation
$$
h(\phi_L(-v))=h(\phi_L(v))+|v|,
$$
 which is more standard in Heegaard Floer literature.

\subsection{Maps between subcomplexes}

Let $z_K(v)$ denote the generator in the homology of $A_{K}^{-}(v)$.
For $K_1\subset K_2$ the complex $A_{K_2}^{-}(v)$ is a subcomplex of the complex
$A_{K_1}^{-}(v)$, so one can define the inclusion maps
$$
j_{K_1,K_2}:A_{K_2}^{-}(v)\hookrightarrow A_{K_1}^{-}(v)
$$
such that for $K_1\subset K_2\subset K_3$ one has $j_{K_1,K_2}\circ j_{K_2,K_3}=j_{K_1,K_3}$.
It is proven in \cite{GN2} that $j_{K_1,K_2}$ does not vanish on homology, in fact,
\begin{equation}
\label{eqn: j}
j_{K_1,K_2}(z_{K_2}(v))=U^{h_{K_2}(v)-h_{K_1}(v)}z_{K_1}(v).
\end{equation}
Set, as above, the dual point $v^*=c(L)-v$. For every $K$, the dual point $(v_K)^*
:= c(L_K)-v_K$ of  $v_K$ is determined in the projected lattice $\BZ^{|K|}$.
From the definition directly follows the next
\begin{lemma}
\label{lem:dual proj}
 $(v_K)^*$  and $v^*$ are related by
$
(v_K)^*=(v^*-\Lambda_{\overline{K}})_K.
$
\end{lemma}
\noindent One can also define another, the `dual' map
$$
j^{\vee}_{K_1,K_2}:A^{-}_{K_2}((u_{K_2})^*)\to A^{-}_{K_1}((u_{K_1})^*)
$$
by the equation
\begin{equation}
\label{eqn: j dual}
j^{\vee}_{K_1,K_2}(z_{K_2}((u_{K_2})^*))=U^{h_{K_2}(u)-h_{K_1}(u)}z_{K_1}((u_{K_1})^*).
\end{equation}
\begin{lemma}
For\,  $K_1\subset K_2$\, the following equation holds:
\begin{equation}
\label{eqn: j dual 2}
j^{\vee}_{K_1,K_2}(z_{K_2}(v))=U^{h_{K_2}\left(v^*-\Lambda_{\overline{K_2}}\right)-h_{K_1}\left(v^*-\Lambda_{\overline{K_2}}\right)}z_{K_1}\left(v-\Lambda_{K_2-K_1}\right).
\end{equation}
\end{lemma}
\begin{proof}
After substitution $u=w^*$ and applying Lemma \ref{lem:dual proj} one transforms \eqref{eqn: j dual} into:
$$
j^{\vee}_{K_1,K_2}(z_{K_2}(w-\Lambda_{\overline{K_2}}))=U^{h_{K_2}(w^*)-h_{K_1}(w^*)}z_{K_1}(w-\Lambda_{\overline{K_1}}).
$$
Now substitute $w=v+\Lambda_{\overline{K_2}}$\, and use
$
w^*=v^*-\Lambda_{\overline{K_2}}$ \, and \ $w-\Lambda_{\overline{K_1}}=v-\Lambda_{K_2-K_1}
$.
\end{proof}

\begin{example}
If $i\in K$ then
$
j^{\vee}_{K-i,K}(z_{K}(v))=U^{h_{K}\left(v^*-\Lambda_{\overline{K}}\right)-h_{K-i}\left(v^*-\Lambda_{\overline{K}}\right)}z_{K-i}\left(v-\Lambda_{i}\right).
$
\end{example}


\subsection{Surgery complex}

The surgery complex is a direct sum $\Cc=\bigoplus_{K,v}A_{K}^{-}(v)$, see Figure \ref{fig1}.  The differential consists of three parts:
internal differential $\partial_{in}$ defined in each $A^{-}_{K}(v)$, and ``short" and ``long" differentials acting between different $A^{-}_{K}(v)$.
The ``short" differential sends $A_{K}^{-}(v)$ to $A_{K-i}^{-}(v)$ via the map $j_{K-i,K}(v)$ for all $i\in K$. The ``long" differential sends $A_{K}^{-}(v)$ to $A_{K-i}^{-}(v-\Lambda_i)$ and is given by the map $j^{\vee}_{K-i,K}(v)$ for all $i\in K$.  We refer to \cite[Lemma 5.5]{Liu} for further details and
for the proof
of the duality between the ``short" and ``long" parts of the differential. The complex decomposes into a direct sum of $|\det \Lambda|$ subcomplexes corresponding to $spin^c$-structures on the surgery manifold $S^{3}_d(L)$. We will write $\partial_{ext}$ for the sum of
``short" and ``long" differentials, so that $\partial=\partial_{in}+\partial_{ext}$.

\begin{remark}
In \cite{Liu,MO} and in the knot surgery formula  (which may be more familiar to experts in Heegaard Floer homology) the ``long" differential shifts the Alexander grading by $\Lambda_i$ rather than $(-\Lambda_i)$. This difference is caused by the equation \eqref{eqn: phi}, which
reverses the direction of all Alexander gradings.
\end{remark}

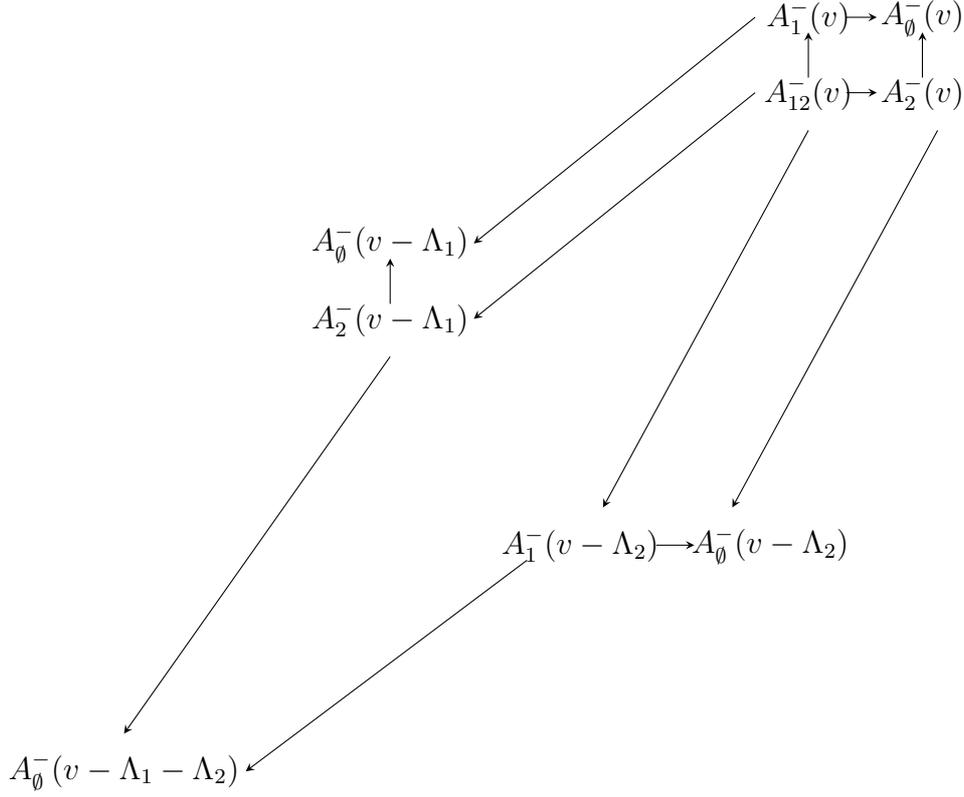
\begin{figure}[ht]
\begin{tikzpicture}
\draw (-0,-0) node {$A_{\emptyset}^{-}(v)$};
\draw (-0,-1) node {$A_{2}^{-}(v)$};
\draw (-1.5,-0) node {$A_{1}^{-}(v)$};
\draw (-1.5,-1) node {$A_{12}^{-}(v)$};
\draw (-2,-7) node {$A_{\emptyset}^{-}(v-\Lambda_2)$};
\draw (-4.5,-7) node {$A_{1}^{-}(v-\Lambda_2)$};
\draw (-7,-3) node {$A_{\emptyset}^{-}(v-\Lambda_1)$};
\draw (-7,-4) node {$A_{2}^{-}(v-\Lambda_1)$};
\draw (-10.5,-10) node {$A_{\emptyset}^{-}(v-\Lambda_1-\Lambda_2)$};

\draw [->,>=stealth] (-1.5,-1.5)--(-4.2,-6.5);
\draw [->,>=stealth] (0.2,-1.5)--(-2.5,-6.5);
\draw [->,>=stealth] (-2.2,-1)--(-5.9,-4);
\draw [->,>=stealth] (-2.2,-0)--(-5.9,-3);
\draw [->,>=stealth] (-5.2,-7.2)--(-8.9,-10);
\draw [->,>=stealth] (-7,-4.5)--(-10.5,-9.5);

\draw  [->,>=stealth] (-1,-0)--(-0.6,-0);
\draw  [->,>=stealth] (-1,-1)--(-0.6,-1);
\draw  [->,>=stealth] (-3.5,-7)--(-3,-7);
\draw  [->,>=stealth] (-0,-0.8)--(-0,-0.2);
\draw  [->,>=stealth] (-1.5,-0.8)--(-1.5,-0.2);
\draw  [->,>=stealth] (-7,-3.8)--(-7,-3.2);
\end{tikzpicture}
\caption{General scheme of Manolescu-Ozsv\'ath surgery complex }
\label{fig1}
\end{figure}

If we take the homology of $A_{K}^{-}(v)$ at each vertex (i. e. with respect to $\partial_{in}$), we get at every place $v$ a copy of $\F[U]$ generated by $z_{K}(v)$. On the homology of $\partial_{in}$ the external differential $\partial_{ext}$ (or $\partial$) induces
 the following differential (since we work over\, $\F$, we ignore the signs):
\begin{equation}
\label{main d}
\partial(z_{K}(v))=\sum_{i\in K}U^{h_K(v)-h_{K-i}(v)}z_{K-i}(v)+
U^{h_{K}\left(v^*-\Lambda_{\overline{K}}\right)-h_{K-i}\left(v^*-\Lambda_{\overline{K}}\right)}z_{K-i}\left(v-\Lambda_{i}\right).
\end{equation}
The complex is absolutely graded, and $\partial$ has homological degree $(-1)$.
 It is important to note that in general this complex may not give the Heegaard-Floer homology due to the presence of the higher differentials.
There is, however, a spectral sequence \cite{Lidman} such that $E_{\infty}=\HF^{-}(S^3_{d}(L))$ and $E_2=H_{*}(H_{*}(\Cc,\partial_{in}),\partial)$.



\begin{theorem}(\cite[Theorem 1.17]{Liu})
For two-component L--space links, the spectral sequence associated with the filtration on the link surgery complex degenerates at $E_2$ page,
hence
$$
\HF^{-}(S^3_{d_1,d_2}(L))\simeq H_{*}(\Cc,\partial)\simeq H_{*}(H_{*}(\Cc,\partial_{in}),\partial).
$$
\end{theorem}

The simplified surgery complex for a two-component L--space link is shown in Figure \ref{surg 2}.
Here $$v^*=c(L)-v=(2g_1+l-v_1,2g_2+l-v_2)$$ and, as above, $g_i$ is the genus of a component $L_i$ and $l$ is the linking number between the components.

\begin{figure}[ht!]
\begin{tikzpicture}
\draw (0,0) node {$z_{\emptyset}(v-\Lambda_1-\Lambda_2)$};
\draw (5,0) node {$z_{\{1\}}(v-\Lambda_2)$};
\draw (0,5) node {$z_{\{2\}}(v-\Lambda_1)$};
\draw (5,5) node {$z_{\{1,2\}}(v)$};
\draw (10,0) node {$z_{\emptyset}(v-\Lambda_2)$};
\draw (10.5,5) node {$z_{\{2\}}(v)$};
\draw (10.5,10) node {$z_{\emptyset}(v)$};
\draw (0,10) node {$z_{\emptyset}(v-\Lambda_1)$};
\draw (5,10) node {$z_{\{1\}}(v)$};

\draw [->,>=stealth] (3.8,0)--(1.5,0);
\draw [->,>=stealth] (6,0)--(9,0);
\draw [->,>=stealth] (4,5)--(1.2,5);
\draw [->,>=stealth] (6,5)--(9,5);
\draw [->,>=stealth] (4,10)--(1,10);
\draw [->,>=stealth] (6,10)--(9,10);
\draw [->,>=stealth] (0,4)--(0,1);
\draw [->,>=stealth] (0,6)--(0,9);
\draw [->,>=stealth] (5,4)--(5,1);
\draw [->,>=stealth] (5,6)--(5,9);
\draw [->,>=stealth] (10,4)--(10,1);
\draw [->,>=stealth] (10,6)--(10,9);

\draw (7.5,10.2) node {$\scriptstyle{h_1(v)}$};
\draw (10.5,7.5) node {$\scriptstyle{h_2(v)}$};
\draw (7.5,5.2) node {$\scriptstyle{h(v)-h_2(v)}$};
\draw (5.9,7.5) node {$\scriptstyle{h(v)-h_1(v)}$};
\draw (2.5,10.2) node {$\scriptstyle{h_1(v^*-\Lambda_2)}$};
\draw (0.7,7.5) node {$\scriptstyle{h_2(v-\Lambda_1)}$};
\draw (0.5,2.5) node {$\scriptstyle{h_2(v^*)}$};
\draw (2.5,5.2) node {$\scriptstyle{h(v^*)-h_2(v^*)}$};
\draw (2.5,0.2) node {$\scriptstyle{h_1(v^*)}$};
\draw (6.1,2.5) node {$\scriptstyle{h(v^*)-h_1(v^*)}$};
\draw (7.5,0.2) node {$\scriptstyle{h_1(v-\Lambda_2)}$};
\draw (10.9,2.5) node {$\scriptstyle{h_2(v^*-\Lambda_1)}$};
\end{tikzpicture}
\caption{Powers of $U$ in the surgery complex for a two-component L--space link}
\label{surg 2}
\end{figure}

In what follows we will need some information about the absolute homological gradings  on the surgery complex.
These can be  reconstructed from \eqref{main d} and the following result.

\begin{lemma}
For any fixed $v$ and arbitrary $u$  the absolute homological gradings of the generators $z_{\emptyset}$ are given by the formula
$$
\deg(z_{\emptyset}(v+\Lambda u))=(u,\Lambda u)+(\gamma,u)+const,
$$
where $\gamma=(2v_i-2g_i+d_i)_{i=1}^{r}$ and the constant depends only on the class of the $spin^c$ structure on $S^3_d(L)$ represented by $v$ (that is, only on the sublattice $(v+\Lambda u)_{u}$).
\end{lemma}

\begin{proof}
Let us abbreviate $\deg(v):=\deg(z_{\emptyset}(v))$. Then, by \eqref{main d},
$$
\partial(z_{\{i\}}(v))=U^{h_i(v)}z_{\emptyset}(v)+U^{h_i\left(v^*-\Lambda_{\overline{\{i\}}}\right)}z_{\emptyset}(v-\Lambda_i),
$$
and two terms in the right hand side have the same homological degree.
By Lemma \ref{lem:dual proj} and \eqref{symmetry} one has:
$$
h_i\left(v^*-\Lambda_{\overline{\{i\}}}\right)=h_i((v_i)^{*})=h_i(2g_i-v_i)=h_i(v_i)+g_i-v_i.
$$
Since $\deg(U)=-2$, it follows that
$$
\deg(v-\Lambda_i)-\deg(v)=2(g_i-v_i)=2g_i-2(e_1,v).
$$
For $u\in \BZ^2$ define $Q(u):=\deg(v+\Lambda u)$.  Then
\begin{equation}\label{d Q}\begin{split}
Q(u-e_i)-Q(u)&=\deg(v+\Lambda u-\Lambda_i)-\deg(v+\Lambda u)\\
&=2g_i-2(e_i,v+\Lambda u)=
2(g_i-v_i)-2(e_i,\Lambda u).\end{split}
\end{equation}
The equations \eqref{d Q} determine the function $Q$ up to an overall constant (depending only on the lattice $(v+\Lambda u)_{u}$).
It remains to notice that the quadratic function
$Q'(u):=(u,\Lambda u)+(\gamma,u)$
satisfies the same identities:
\begin{equation*}\begin{split}
Q'(u-e_i)-Q'(u)&=-2(e_i,\Lambda u)+(e_i,\Lambda e_i)-(\gamma,e_i)\\
&= -2(e_i,\Lambda u)+d_i-\gamma_i=2(g_i-v_i)-2(e_i,\Lambda u),\end{split}
\end{equation*}
hence $Q(u)=Q'(u)+const$.
\end{proof}

\begin{corollary}
\label{degree squares}
The absolute homological gradings of the generators $z_{K}(v)$ are given by
\begin{equation*}\begin{split}
\deg(z_K(v+\Lambda u))&=\deg(z_{\emptyset}(v+\Lambda u))-2h_K(v+\Lambda u)+|K|\\
&=(u,\Lambda u)+(\gamma,u)-2h_K(v+\Lambda u)+|K|+const.\end{split}
\end{equation*}
\end{corollary}

\begin{proof}
We prove by induction on $|K|$ that
\begin{equation}
\label{eq: deg z}
\deg(z_K(v))=\deg(z_{\emptyset}(v))-2h_K(v)+|K|.
\end{equation}
For $K=\emptyset$ the equation is clear.
Assume that \eqref{eq: deg z} holds for $K-i$, then \eqref{main d} implies:
$$
\deg(z_K(v))-1=\deg(z_{K-i}(v))-2h_{K}(v)+2h_{K-i}(v),
$$
so \eqref{eq: deg z} holds for $K$ as well.
\end{proof}

\subsection{Very good points and bounded surgeries}

From now on we consider only links with two components.

\begin{definition}
\label{very good}
Let us call a lattice point $v=(v_1,v_2)\in \BZ^2$ {\em good} for an L--space link $L$, if $h(v)>h_1(v_1)$ and $h(v)>h_2(v_2)$,
and {\em very good} for $L$, if both $v$ and $v^*$ are good for $L$.
\end{definition}

The following theorem is one of the main results of the article.

\begin{theorem}[Theorem \ref{th:bounded}]\label{th:bounded2}
Suppose that there exists a very good point for an L--space link $L$. Then for all L--space
surgeries on $L$ the framing matrix $\Lambda$ is positive definite.
\end{theorem}

\begin{proof}
Suppose that $v$ is a very good point for $L$.
Consider the surgery complex for a $(d_1,d_2)$--surgery of $S^3$ along $L$ with $spin^c$ structure,
corresponding to $v$. Since $v$ is very good, all four numbers
$h(v)-h_1(v),h(v)-h_2(v),h(v^*)-h_1(v^*),h(v^*)-h_2(v^*)$ are strictly positive, hence
 the boundary $\partial(z_{\{1,2\}}(v))$ is divisible by $U$.
Consider the cycle
$$
Z(v):=U^{-1}\partial(z_{\{1,2\}}(v)).
$$
One has $\partial Z(v)=0$ and $UZ(v)\in \Imm\partial$.
Since $S^3_d(L)$  is an L--space, its homology $H_{*}(\Cc,\partial)$ is  isomorphic to $\BZ[U]$,
hence it has no nontrivial element annihilated by $U$. Hence,
one should have $Z(v)=\partial \alpha$.
 Such an $\alpha$  must have the form
\begin{equation}
\label{eq: alpha}
\alpha=\sum_{u\in \BZ^2\setminus (0,0)}  U^{N(u)}z_{\{1,2\}}(v+\Lambda u)
\end{equation}
for some $N(u)\ge 0$ (otherwise $\partial \alpha$ would contain more terms).

Let us compare the homological degrees in \eqref{eq: alpha}. By Corollary \ref{degree squares} we have
\begin{equation*}\begin{split}
\deg(\alpha)& =\deg(z_{\{1,2\}}(v+\Lambda u))-2N(u)\le \deg(z_{\{1,2\}}(v+\Lambda u))\\
& =(u,\Lambda u)+(\gamma,u)-2h(v+\Lambda u)+2+const\le (u,\Lambda u)+(\gamma,u)+2+const.
\end{split}\end{equation*}
We conclude that the quadratic form $Q(u)=(u,\Lambda u)+(\gamma,u)$ is
bounded from below on $\BZ^2\setminus \{(0,0)\}$. Since this happens for any $v$ (hence any $\gamma$),
$\Lambda$ must be positive definite.
\end{proof}

\begin{remark}
One can apply a similar argument for knots, where the very good points can be defined by inequalities
$$
h(v)>0, \   h(v^*)=h(2g-v)>0.
$$
However, for any L--space knot $h(v)>0$ if and only if $v>0$, and $h(v^*)>0$ if and only if $v<2g$.
Therefore for any nontrivial L--space knot all points $v\in [1,2g-1]$ are very good.
Similarly to Theorem \ref{th:bounded}, one proves that all L--space surgeries on a
{\em nontrivial} L--space knot are positive (cf. Theorem \ref{th: ls knots}).
\end{remark}

\begin{remark}
At present, we cannot generalize Theorem \ref{th:bounded} to the case of links with 3 or more components, since the cycle
$Z(v)$ may be annihilated by a higher differential.
\end{remark}

\section{Negative definite graph manifolds and their surgeries}
\label{ss:graph mfd}

\subsection{}
Consider a rational homology sphere
graph manifold $Y$ corresponding to a negative definite plumbing graph $\Gamma$.
Each vertex $v$ defines a knot $K_v$ in $Y$.
The pair  $(\Gamma,v)$ and an integer $d'$ determine another plumbing graph  $\Gamma_{d'}$ constructed as follows.
We add to the graph $\Gamma$ (whose shape and decorations we keep) another new vertex,
 say $v_{new}$, with decoration $d'$ (and genus zero),
which is connected to vertex $v$ by an edge.
Then $\det(\Gamma_{d'})=-d'\det(\Gamma)-\det(\Gamma\setminus v)$, denoted by $-d$
(recall that $\det(\Gamma)=\det(-I_\Gamma)$). Hence
$\Gamma_{d'}$ is a negative definite graph whenever $d<0$.
This graph represents the
surgery 3--manifold $Y_{d}(K_v)$. If $d\not=0$ then $Y_{d}(K_v)$ is a rational homology sphere with
$|H_1(Y_{d}(K_v))|=|d|$. (This is compatible with the construction from subsection \ref{ss:gm},
where $\det(\Gamma)=1$ and $\det(\Gamma\setminus v_i)=m_i$.)

If $\Gamma $ is any connected negative definite graph with vertices $\Vv$ and plumbed 4--manifold
$P(\Gamma)$, then its lattice $L$ is $H_2(P(\Gamma),\BZ)$ with intersection form $(\cdot,\cdot)$.
If $\{E_v\}_{v\in\Vv}$ denote the cores in $P(\Gamma)$, then $L=\BZ\langle E_v\rangle_v$, and the
intersection form $I$ associated with $\Gamma$ is exactly $(E_v,E_w)_{v,w}$.

The Lipman cone in $L$ is defined (see e.g. \cite{Lipman,NGr}) by
 \begin{equation}\label{eq:conegamma}
C(\Gamma)=\left\{Z \in L\ :\   (Z,E_v)\le 0\ \ \mbox{for all $v\in\Vv$}\right\}.
\end{equation}
The {\it minimal} (or fundamental) cycle $Z_{\min}=Z_{\min}(\Gamma)$ of $\Gamma$
  is the unique non--zero minimal element in $C(\Gamma)$, cf. \cite{Artin62,Artin66}.
%
It is known (e.g. \cite{Artin66}) that if $Z=\sum_vn_vE_v\in C(\Gamma)$
then $n_v\ge 0$ for all $v$, and if additionally $Z\not=0$ then $n_v> 0$ for all $v$.
In particular $\sum_vE_v\leq Z_{\min}\leq Z$ for any $Z\in C(\Gamma)$.

The minimal cycle can be used to define {\em rational graphs} via
{\it Laufer's Rationality Criterion}  \cite{Laufer72}.
First we recall Laufer's algorithm, whose output is the minimal cycle. This provides a
{\it computation sequence}
$\{z_i\}_{i=0}^t\in L$, such that $z_0$ is one of the
arbitrarily chosen base elements $E_v$, $z_t=Z_{\min}$, and $\{z_i\}_{i=0}^t$ is constructed
inductively as follows \cite{Laufer72}.
Assume that $z_i$ was already constructed. If $(z_i,E_v)\leq 0$ for all $v$ then we stop:
 $i=t$, and $z_i=Z_{\min}$. If $(z_i,E_v)> 0$ for a certain $v$, then choose one of such vertices, say
 $v(i)$, and  set $z_{i+1}=z_i+E_{v(i)}$,
 and restart the algorithm again. The procedure necessarily stops after finitely many
  steps, and the final $z_t$ is always
 $Z_{\min}$ (though the sequence is not necessarily unique).

 Then, {\em Laufer's Rationality Criterion} says  that
 $\Gamma$ is rational if and only if along an arbitrarily chosen computation sequence
 (hence along all the computation sequences)
 at every step $i<t$ one has $(z_i,E_{v(i)})= 1$, see \cite{Laufer72}.
(We will call the integers $(z_i,E_{v(i)})$ `testing numbers'.)

It is not hard to verify using this criterion that rational graphs are stable by taking subgraphs or by
decreasing the decorations of a graph. In both these two cases one can construct a computation sequence
for a subgraph, or for a modified graph with
decreased  decorations, which is the (starting) part  of a computation sequence
of the original graph.


\begin{definition}
\label{def simple} Fix a connected negative definite graph $\Gamma$.
A vertex $v$ of $\Gamma$ is called {\it simple} if the coefficient of $E_v$ in $Z_{\min}(\Gamma)$ equals 1.
\end{definition}
Since $\sum_vE_v\leq Z_{\min}\leq Z$ for any $Z\in C(\Gamma)$, $v$ is simple if and only if
there exists $Z\in C(\Gamma)$, whose $E_v$--coefficient is 1.

The following theorem describes when the set of L--space surgeries of $Y$ along $K_v$ is bounded.
(Recall, see Theorem \ref{th:ratL}, that
a negative definite graph $\Gamma$ defines an L-space  if and only if $\Gamma$ is rational.)
\begin{theorem}
\label{th: l space simple vertex}
Assume that  $\Gamma$ is  a negative definite
 rational graph (so $Y$ is an L--space). Then the following statements hold:

a) For $d\gg 0$, $Y_{d}(K_v)$ is an L--space.

b) For $d\ll 0$, $Y_{d}(K_v)$ is an L--space if and only if $v$ is a simple vertex of\, $\Gamma$.
\end{theorem}

\begin{proof}
Note that  $d\gg 0$ (resp. $d\ll 0$) if and only if $d'\gg 0$ (resp. $d'\ll 0$).
The proof of (a) is identical to the proof of the main theorem of \cite{GN1}. Next we prove (b).
Since for $d'\ll 0$ the graph $\Gamma_{d'}$ is negative definite, the statement transforms into the rationality of
 $\Gamma_{d'}$.

 Let $n$ denote the coefficient of $E_v$ in $Z_{\min}$.

Suppose that $\Gamma_{d'}$ is rational. Let us  run Laufer's algorithm for $Z_{\min}(\Gamma_{d'})$
in such a way that $z_0$ is a base element of $L(\Gamma)$ and at all steps
we choose $E_{v(i)}$ from the support of $\Gamma$ whenever is possible.
Then at an intermediate steps we have $x_i=Z_{\min}(\Gamma)$. The next choice is necessarily
$E_{v(i)}=E_{new}$, and the Laufer's testing number is $(x_i, E_{new})=(Z_{\min}(\Gamma),E_{new})=n$.
Hence $n=1$ by Laufer's Criterion. See also \cite[Corollary 4.1]{LeTosun}.

In fact, we proved the following {\it general fact}: if $\Delta$ is a subgraph of a rational graph
$\Delta'$, and $(v,v')$ is an edge in $\Delta'$ such that $v\in \Delta$ but $v'\not\in\Delta$, then the
$E_v$--coefficient of $Z_{\min}(\Delta)$ is 1.

Conversely, assume that $n=1$, and we prove that $\Gamma_{d'}$ is rational for $d'\ll 0$. This essentially follows from \cite[Theorem 4.8]{LeTosun}, but we present here a slightly shorter proof (adopted to this situation)
for the reader's convenience. Following \cite{Sp,Ty} we introduce some notations.

For any graph $G$, we say that $u\in \Vv(G)$ is a {\em Tjurina vertex} of $G$
if $(Z_{\min}(G),E_u)=0$.

Let $\Delta_1$ be the connected component of the set of Tjurina vertices of $\Gamma$ (as full subgraph),
which contains $v$ (if $v$ is not a Tjurina vertex, $\Delta_1=\emptyset$). $\Delta_1\subsetneq\Gamma$ since $(Z_{\min}(\Gamma),E_u)$ cannot be zero for all $u\in\Vv(\Gamma)$.
Let $\Delta_2$ be the connected component of the set of Tjurina vertices of $\Delta_1$, which contains $v$ etc. By repeating this procedure, we obtain a sequence of properly nested subgraphs:
$$
\Gamma \supsetneq \Delta_1 \supsetneq \Delta_2 \supsetneq\ldots \supsetneq \Delta_k=\emptyset.
$$
We claim that if $d'\leq -k$ and $\Gamma_{d'}$ is negative definite then $\Gamma_{d'}$ is rational. Indeed, let us run the Laufer's algorithm for
$\Gamma_{d'}$. We start with $z_0=E_{new}$, hence $z_1=E_{new}+E_v$.
 Then the  next few steps are identical with the steps of the algorithm for $\Gamma$, hence at some point
  we obtain the  cycle $z_s=E_{new}+Z_{\min}(\Gamma)$.
  (The assumption is used here: since the $E_v$--multiplicity in $Z_{\min}(\Gamma)$ is 1, during the steps between
  $E_{new}+E_v$ and $E_{new}+Z_{\min}(\Gamma)$ we do not need to add $E_v$, hence we never test for
  $(x_i,E_v)$, which is changed by the presence of $E_{new}$.)
  Note that $(z_s,E_u)=(Z_{\min},E_u)\le 0$ for $u\in\Vv\setminus v$ and $(z_s,E_{new})=1+d'\leq 0$.  If $v$ is not a Tjurina vertex for $\Gamma$,
  we have $(Z_{\min}(\Gamma),E_v)\le -1$, hence $(z_s,E_v)\le 0$,
and the algorithm stops, $Z_{\min}(\Gamma_{d'})=E_{new}+Z_{\min}(\Gamma)$ with testing numbers  1 along all the steps.
If $v$ is a Tjurina vertex for $\Gamma$, we need to continue with
$z_{s+1}=z_s+E_v$ (whose testing number is 1 again).
Then along the next few  steps we choose   $v(i)$  imposed by the
algorithm of $Z_{\min}(\Delta_1)$. Hence, we will arrive at the cycle
$
z_{s'}=E_{new}+Z_{\min}(\Gamma)+Z_{\min}(\Delta_1)$.
This cycle satisfies $(z_{s'},E_u)\le 0$ for $u\in \Vv\setminus v$ (even for $v\in \Vv(\Gamma\setminus \Delta_1)$ thanks to the above {\it general fact} regarding subgraphs of rational graphs, applied for the pair
$\Delta_1\subset \Gamma$).
Furthermore, $(z_{s'},E_{new})=2+d'\leq 0$ too (since $Z_{\min}(\Delta_1)\leq Z_{\min}(\Gamma)$, hence both have
$E_v$--coefficient 1). Thus the  only vertex that eventually
needs correction is $v$. Note that again
$
(z_{s'},E_v)=(E_{new},E_v)+(Z_{\min}(\Gamma)+Z_{\min}(\Delta_1),E_v)\le 1.
$
We repeat this procedure until we get the cycle
$$
z_{t}=E_{new}+Z_{\min}(\Gamma)+Z_{\min}(\Delta_1)+\ldots+Z_{\min}(\Delta_{k-1}).
$$
Then, $(x_t,E_u)\leq $ for all vertices of $\Gamma_{d'}$, hence
 $z_t=Z_{\min}(\Gamma_{d'})$.
 Since along all the steps the testing numbers $(x_i,E_{v(i)})=1$,  $\Gamma_{d'}$ is rational.
\end{proof}



\begin{example}
Consider the plumbing graph for the lens space $L(10,9)$ (or $A_9$ singularity) shown in Figure \ref{fig:gap}
(nine $(-2)$--vertices).
Its minimal  cycle has coefficient 1 at each vertex.
One can check that a $d$-surgery on its central vertex $v$ is an $L$-space if and only if $d'\in (-\infty,-4]\cup [-1,+\infty)$.
 The rectangles represent the subgraphs $\Delta_i$ appearing in the proof of Theorem \ref{th: l space simple vertex}.
(Note that $\Gamma_{-3}$ is negative definite but not rational.)
\end{example}

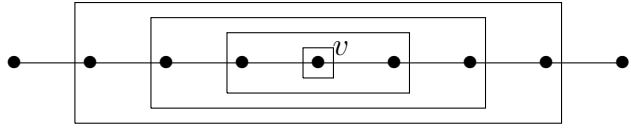
\begin{figure}
\begin{tikzpicture}
\draw (0,0)--(8,0);
\draw (0,0) node {$\bullet$};
\draw (1,0) node {$\bullet$};
\draw (2,0) node {$\bullet$};
\draw (3,0) node {$\bullet$};
\draw (4,0) node {$\bullet$};
\draw (5,0) node {$\bullet$};
\draw (6,0) node {$\bullet$};
\draw (7,0) node {$\bullet$};
\draw (8,0) node {$\bullet$};
\draw (3.8,-0.2)--(3.8,0.2)--(4.2,0.2)--(4.2,-0.2)--(3.8,-0.2);
\draw (2.8,-0.4)--(2.8,0.4)--(5.2,0.4)--(5.2,-0.4)--(2.8,-0.4);
\draw (1.8,-0.6)--(1.8,0.6)--(6.2,0.6)--(6.2,-0.6)--(1.8,-0.6);
\draw (0.8,-0.8)--(0.8,0.8)--(7.2,0.8)--(7.2,-0.8)--(0.8,-0.8);
\draw (4.3,0.2) node {$v$};
\end{tikzpicture}
\caption{Resolution of $A_9$ singularity and the subgraphs $\Delta_i$}
\label{fig:gap}
\end{figure}

\begin{remark}\label{rem:analsimple}
{\it (Analytic interpretation of simple vertices.)} \ Assume that $(X,o)$ is a rational
complex normal surface singularity (that is,  its geometric genus is zero, or equivalently, any of its
good resolution graphs is rational). Let $(C,o)\subset (X,o)$ be an irreducible curve in it. Assume that
$\Gamma$ is the resolution graph of a good embedded resolution $\widetilde{X}\to X$ (that is,
the total transform of $C$ is a normal crossing divisor). Let $E_v$ be the irreducible exceptional curve,
which intersects the strict transform of $(C,o)$. Then the vertex $v$ is simple if and only if
$(C,o)$ is smooth.
Indeed, for rational singularities the pull--back of the maximal ideal of ${\mathcal O}_{X,o}$
is ${\mathcal O}_{\widetilde{X}}(-Z_{\min})$ and it has no basepoint \cite{Artin62,Artin66}.
Hence, the multiplicity of $(C,o)$ (that is, the intersection of $(C,o)$ with a generic linear form)
is the $E_v$--multiplicity of $Z_{\min}(\Gamma)$. But ${\rm mult}_o(C,o)=1$ if and only if $(C,o)$ is smooth.
\end{remark}

\section{Invariants of algebraic links}\label{s:alglinks}

\subsection{Semigroup, Alexander polynomial and the $h$--function}\label{ss:algfacts}
Let $C=C_1\cup C_2\subset (\BC^2,0)$ be a  plane curve singularity with 2 components.
Let $L=L_1\cup L_2\subset S^3$ be the corresponding link. Let $\gamma_i:(\BC,0)\to (C_i,0)$ be the
normalization of $C_i$.

\begin{definition}For any function $f\in \BC\{x,y\}$ set
$\nu_i(f)=\Ord_t(f(\gamma_i(t)))$.
The semigroup $S_C$ of the germ $C$ is the set of  pairs $(\nu_1(f),\nu_2(f))\in (\BZ_{\geq 0})^2$ for all $f\in \BC\{x,y\}$.
\end{definition}

One defines similarly the semigroup of a  one-component curve. If $C_1$ is a component of
 $C=C_1\cup C_2$ then $S_{C_1}$ is  the image of the first projection of $S_C$.

In the next proposition $K$ is an algebraic {\it knot},  $S$ is the semigroup of the corresponding curve--germ,
$\Delta(t)$ is  the Alexander {\it polynomial} of $K$. It is well--known that the degree $\mu$ of $\Delta(t)$
is twice the genus of $K$.
\begin{proposition}\label{alex 1 component} \cite{CDG}
With these notations  the following statements hold.
For $s\ge \mu$ one has $s\in S$ (in fact, $\mu$ is optimal with this property, that is,
$\mu$ is the conductor of $S$),  and $\sum_{s\in S}t^s=\Delta(t)/(1-t)$.
\end{proposition}
\begin{corollary}
\label{cor:zeta with a factor}
For any $M\ge \mu$
all the coefficients of the polynomial $\Delta(t)\cdot \frac{1-t^M}{1-t}$ are equal to 0 or 1.
If $\mu\leq s<M$ then the coefficient at $t^s$ in this polynomial equals 1.
\end{corollary}
\subsection{}
We will also need the following  facts about two--component algebraic links
(see \cite{CDG, GN2} and references therein):
\begin{enumerate}
\item The topologically defined $h$--function (cf. \ref{ss:LFH})
of an algebraic link coincides with the (analytic) Hilbert function of $C$ and it is
determined by the semigroup as follows: $h(v_1+1,v_2)=h(v_1,v_2)+1$, if there exists $u\in S_C$
such that $u_1=v_1$ and $u_2\ge v_2$. Otherwise $h(v_1+1,v_2)=h(v_1,v_2)$. The difference $h(v_1,v_2+1)-h(v_1,v_2)$ can be described in a similar way.
\item
If $u,v\in S_C$ then ${\rm inf}(u,v)\in S_C$ as well. Hence
$v\in S_C$ if and only if
$$
h(v_1+1,v_2)=h(v_1,v_2+1)=h(v_1,v_2)+1.
$$
\item A coefficient $a_v$ ($v=(v_1,v_2)$) of $t_1^{v_1}t_2^{v_2}$
in the Alexander polynomial equals
$$
a_v=h(v_1+1,v_2)+h(v_1,v_2+1)-h(v_1,v_2)-h(v_1+1,v_2+1).
$$
Using the above description of $h(v)$, one can check that
$$
a_v=\begin{cases}
1\ \text{if}\ h(v_1+1,v_2)=h(v_1,v_2+1)=h(v_1+1,v_2+1)=h(v_1,v_2)+1,\\
0\ \text{otherwise}.\\
\end{cases}
$$
\item In particular, if $a_v=1$ (so $v\in \Supp(\Delta)$) then $v$ belongs to the semigroup of $C$.
Furthermore, $v\in  \Supp(\Delta)$ if and only if $v\in S_C$, and $S_C\cap \{(v_1,u_2)\,:\, u_2>v_2\}=
S_C\cap \{(u_1,v_2)\,:\, u_1>v_1\}=\emptyset$. This also shows that $\Supp(\Delta)$
cannot have distinct  pairs $u,v$ with $u_1=v_1$ or with $u_2=v_2$.
\item Using (\ref{symmetry}), $v\in \Supp(\Delta)$ if and only if
$v\in S_C$ and $v^*-{\bf 1}\in S_C$.  (Here ${\bf 1}=(1,1)$.) Hence $v\in \Supp(\Delta)$
if and only if $c-{\bf 1}-v\in \Supp(\Delta)$.
\item
$h(v)=h({\rm sup}(v,(0,0)))$, and
the $h$--functions for the components of $L$ are given by: $$h_1(v)=h(v_1,0), \ \ h_2(v)=h(0,v_2).$$
\end{enumerate}

\begin{lemma}
\label{alg good}
A point $v=(v_1,v_2)$ is good for an algebraic link $L$ if and only if there exist semigroup points
$$a\in [v_1,+\infty)\times [0,v_2-1]\ \text{and}\ b\in  [0,v_1-1]\times [v_2,+\infty).$$
\end{lemma}

\begin{proof}
Consider the difference
$$
h(v)-h_1(v)=h(v_1,v_2)-h(v_1,0)=\sum_{j=0}^{v_2-1}(h(v_1,j+1)-h(v_1,j)).
$$
In the last sum each summand is either equal to 0 or to 1, hence $h(v)-h_1(v)>0$
if and only if $h(v_1,j+1)-h(v_1,j)=1$ for at least one  $j\in [0,v_2-1]$. The latter equation holds if
there is a semigroup point $a=(a_1,a_2)$ such that $a_1\ge v_1$ and $a_2=j$.
\end{proof}

\begin{lemma}
\label{alg very good}
If the Alexander polynomial $\Delta$ is not of ordered type then there is a very good point for $L$.
\end{lemma}

\begin{proof}
Suppose that the Alexander polynomial $\Delta$ is not of ordered type.
This means that there are points $u=(u_1,u_2),v=(v_1,v_2)\in \Supp(\Delta)$ such that
$u_1<v_1$ but $u_2>v_2$.

Since $u$ and $v$ are both in the semigroup, by Lemma \ref{alg good}
all points $w$ satisfying
\begin{equation}
\label{eq w}
\inf(u,v)+{\bf 1}\preceq w\preceq \sup(u,v)
\end{equation}
are good.
Furthermore, by the symmetry of $\Delta$,
the points $c-{\bf 1}-u$ and $c-{\bf 1}-v$ belong to its support too, and clearly
$$
\inf(c-{\bf 1}-u,c-{\bf 1}-v)+{\bf 1}\preceq c-w\preceq \sup(c-{\bf 1}-u,c-{\bf 1}-v),
$$
hence $w^*=c-w$ is a good point too. Therefore any $w$ satisfying \eqref{eq w} is very good.
\end{proof}


\begin{lemma}
\label{alex on a line}
Suppose that $0<v_1<l:=\lk(L_1,L_2)$ and $v_1$ belongs to the semigroup of $C_1$. Then there exists $v_2>0$
such that $(v_1,v_2)\in \Supp(\Delta)$.
\end{lemma}

\begin{proof}
By Torres  formula \cite{Torres}
\begin{equation}
\label{torres}
\Delta(t,1)=\frac{\Delta_1(t)}{1-t}\cdot(1-t^l),
\end{equation}
where $\Delta_1(t)$ is the Alexander polynomial of $L_1$.
By Proposition \ref{alex 1 component}  the coefficient of $t^{v_1}$ in $\frac{\Delta_1(t)}{1-t}$  equals $1$. Since $v_1<l$,
the coefficient of $t^{v_1}$ in polynomial from  the right hand side of \eqref{torres} equals 1 as well.
But this number, read from the left hand side of \eqref{torres}, is
$$
\sum_{v_2> 0}a_{v_1,v_2}=|\,\{v_2:(v_1,v_2)\in \Supp(\Delta)\}\,|.
$$

\vspace*{-10mm}

\end{proof}

\subsection{The Alexander polynomial from resolution graphs}
\label{ss:ordered to min}




Let  $\Gamma$ be the dual graph (with non--arrowhead vertices $\Vv$ and two arrowheads)
of a good  embedded resolution of $(C,0)\subset (\BC^2,0)$. Let $I$ be the intersection matrix and define
$m_{vw}$ as the $(v,w)$--entry of $-I^{-1}$.
It is well known that $m_{vw}\ge 0$ (see also \ref{BraunN}(b) below).
If $v_1$ and $v_2$ support the arrowheads corresponding to the link components, and $\delta_w$ denotes the valency
of the non--arrowhead vertex $w$ (including the arrowhead supporting edges) then  (see e.g. \cite{EN})
\begin{equation}
\label{eq:EN}
\Delta(t_1,t_2)=\prod_{u\in \Vv}(1-t_1^{m_{uv_1}}t_2^{m_{uv_2}})^{\delta_u-2}.
\end{equation}
Sometimes (for brevity) we use splice diagrams instead of resolution (for their definition, properties and
equivalence with resolution graphs, see \cite{EN}). They can be obtained as follows: one
 erases all two--valent vertices from $\Gamma$ and
write on the $u$--end  of an edge $(u,v)$ of the resulting graph the determinant of the connected component of $\Gamma-u$
containing $v$ (see also figures below).
By Lemma \ref{BraunN}(b) and \eqref{eq:EN} this data is sufficient to recover the Alexander polynomial
from the splice diagram (see also \cite{EN}).

\subsection{Determinantal properties of resolution graphs}\label{ss:detprop}
We will need several arithmetical properties of the multiplicities $m_{vw}$ (and of the
decorations of the splice diagrams). We list here some of them.
Recall that by our convention $\det(G)=\det(-I_G)$ and $\det(\emptyset)=1$.
Hence $\det(G)>0$ for any subgraph $G$ of
$\Gamma$. Moreover $\det(\Gamma)=1$.

Consider a decomposition of a negative definite connected  graph $\widetilde{G}$ (with no arrowheads)
shown in Figure \ref{fig:BraunN}, and let
 $\overline{uv}$ denote the shortest path in  $G$ connecting  $u$ and $v$.
 (If $G$ is merely an edge then its determinant is 1.)
 Set also
$$
\det(G')=a,\det(G\cup G''\cup v)=p, \det(G'')=p', \det(G\cup G'\cup u)=a', \det(G-\overline{uv})=g.$$

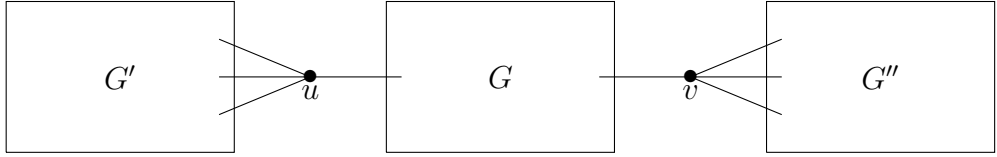
\begin{figure}[ht!]
\begin{tikzpicture}
\draw (0,0)--(0,2)--(3,2)--(3,0)--(0,0);
\draw (4,1) node {$\bullet$};
\draw (5,0)--(5,2)--(8,2)--(8,0)--(5,0);
\draw (9,1) node {$\bullet$};
\draw (10,0)--(10,2)--(13,2)--(13,0)--(10,0);
\draw (2.8,1.5)--(4,1);
\draw (2.8,1)--(4,1);
\draw (2.8,0.5)--(4,1);
\draw (5.2,1)--(4,1);
\draw (7.8,1)--(9,1);
\draw (10.2,1.5)--(9,1);
\draw (10.2,1)--(9,1);
\draw (10.2,0.5)--(9,1);
\draw (1.5,1) node {$G'$};
\draw (6.5,1) node {$G$};
\draw (11.5,1) node {$G''$};
\draw (4,0.8) node {$u$};
\draw (9,0.8) node {$v$};
\end{tikzpicture}
\caption{Decomposition of $\Gamma$ in Lemma \ref{BraunN}}
\label{fig:BraunN}
\end{figure}

Part (a)  of the next Lemma is proved in  \cite[Lemma 4.0.1]{BraunN}, part
(b) in \cite{EN}.
\begin{lemma}\label{BraunN}
(a) \ $\det(G)\cdot\det(\widetilde{G})=a'p-ap'g^2$. \\
\hspace*{2.5cm} (b) \ If\, $\widetilde{G}=\Gamma$ then $m_{uv}=\det(\Gamma-\overline{uv})=ap'g$.
\end{lemma}


\begin{lemma}
\label{lem: inequalities for z and w}
Consider again Figure \ref{fig:BraunN} with $a,p,a',p'$ as above.
 Assume that $\det(\widetilde{G})=1$
and $G-\overline{uv}=\emptyset$ (so $g=1$).
Then there exists positive integers $z$ and $w$ such that
\begin{equation}
\label{eq: cone plus quadrant}
\frac{a'}{p'}>\frac{z}{w}>\frac{a}{p},
\end{equation}
i.e.,  $(zp,zp')$ and $(wa,wa')$ are not comparable with respect to
the partial order of\, $\BZ^2$.
Additionally,
 \begin{equation}\label{eq:3cases}
\left\{ \begin{array}{l}
\mbox{(a) \ if $E_u^2=-1$ and $G'$ is connected then $z<a$}, \\ 
\mbox{(b) \ if $E_v^2=-1$ and $G''$ is connected then $w<p'$}. 
\end{array}\right.\end{equation}
\end{lemma}
\begin{proof}
Let $u',v'$ be the neighbors of $u$ and $v$ in $G$, respectively (they may coincide). Set
$
z_1=\det(G'\cup u)$ and $ w_1=\det(G''\cup v\cup G\setminus u')$.
If we apply Lemma  \ref{BraunN}(a) to $u'$ and $v$ we  get
$$
a'w_1-z_1p'=\det(G-u')>0\ \Rightarrow \frac{a'}{p'}>\frac{z_1}{w_1}.
$$
If we  apply Lemma \ref{BraunN}(a) to $u'$ and $u$ we obtain
$$
pz_1-aw_1=\det(\emptyset)=1>0\ \Rightarrow \frac{z_1}{w_1}>\frac{a}{p}.
$$
By similar computation for
$
w_2=\det(G''\cup v)$ and $ z_2=\det(G'\cup u\cup G\setminus v')
$
we get that
 both  pairs $(z_1,w_1)$ and $(z_2,w_2)$ satisfy  \eqref{eq: cone plus quadrant}.

In the situation of \eqref{eq:3cases}(a),
if $u''$ is the neighbour of $u$ in $G'$ then \eqref{BraunN}(a)
applied for $\widetilde{G}=G'\cup u$ gives
$z_1=a-\det(G'-u'')<a$. Hence $(z_1,w_1)$ satisfies all wished properties.

In case (b)  similarly $w_2<p'$, 
hence $(z_2,w_2)$ satisfies the needed properties.

Finally, assume that both assumptions of (a) and (b) are satisfied.
Then, if $w_1<p'$ then $(z_1,w_1)$ satisfies all requirements,
if $z_2<a$ then $(z_2,w_2)$ works; and
 if $z_2\ge a$ and $w_1\ge p'$ then $z_2\geq a>z_1$ and $w_1\geq p'>w_2$
and $(z_1,w_2)$ is the right choice.
\end{proof}

\section{Links with a trivial  component}\label{s:Trivi}

\subsection{}
In Figures \ref{fig: t46} and \ref{fig: ls two cusps} the sets $\ls(L)$ do not contain
points where the surgery coefficients have large absolute values of opposite signs.
The following results confirms that this is typical for $\ls(L)$.

\begin{theorem}
\label{unknot}
Suppose that $L\subset S^3$ is an L--space link with two components, $d_1\gg 0, d_2\ll 0$ and  $S^{3}_{d_1,d_2}(L)$ is an L--space.
Then $L_2$ is an unknot.
\end{theorem}

\begin{proof}
By \cite[Theorem 1.10]{Liu} both components $L_1$ and $L_2$ are L--space knots.
Consider the 3-manifold $Y=S^3_{d_2}(L_2)$, then $S^{3}_{d_1,d_2}(L)=Y_{d_1}(L_1)$ is a large surgery
on $Y$ along a knot $L_1$. By \cite{OSsurg2,MO} if $Y_{d_1}(L_1)$ is an L--space for $d_1\gg 0$, then
$Y$ itself is an L--space. Hence $S^3_{d_2}(L_2)$ is an L--space.
Suppose that $L_2$ is nontrivial. Then by Theorem \ref{th: ls knots}
$S^3_{d_2}(L_2)$ is an L--space if an only if $d_2\ge 2g(L_2)-1$, which contradicts $d_2\ll 0$.
\end{proof}

For algebraic links we have the following complete characterization.

\begin{theorem}\label{th:corner}
Suppose that $L$ is an algebraic link with two components associated with the curve singularity
$(C,0)\subset (\BC^2,0)$. Then the following facts are equivalent.

(1)  $L_2$ is an unknot, or equivalently, $(C_2,0)$ is smooth;

(2)  $(d_1,d_2)\in \ls(L)$  for any  $d_1\gg 0$ and  $d_2\ll 0$;

(3)  $(d_1,d_2)\in \ls(L)$  for any  $d_1\geq m_1$ and  $d_2\ll 0$;

(4) if\, $\Gamma$ is the embedded resolution graph of $(C,0)\subset (\BC^2,0)$, and
$v_2$ supports the arrowhead of $L_2$ then $v_2$ is simple vertex of $\Gamma$.
\end{theorem}

\begin{proof} $(2)\Rightarrow(1)$ follows from Theorem \ref{unknot},  $(3)\Rightarrow(2)$
is evident,  $(4)\Leftrightarrow(1)$ follows from Remark \ref{rem:analsimple}. Hence we have to prove
$(1)\Rightarrow(3)$.
We proceed similarly to the proof of the main theorem in \cite{GN1}.
We can assume that the self--intersection of $v_1$ in $\Gamma$  is $-1$.
 Recall from \ref{ss:gm} that the graph of the surgery manifold is obtained from $\Gamma$
 by adding two additional vertices with framings $d_1-m_1$ and $d_2-m_2$ glued to the vertices $v_1$ and $v_2$.
Assume first that $d_1>m_1$.
 By plumbing calculus we can replace the first new vertex by the chain of $(-2)$--vertices. Let us call the resulting graph by $\Gamma'$.

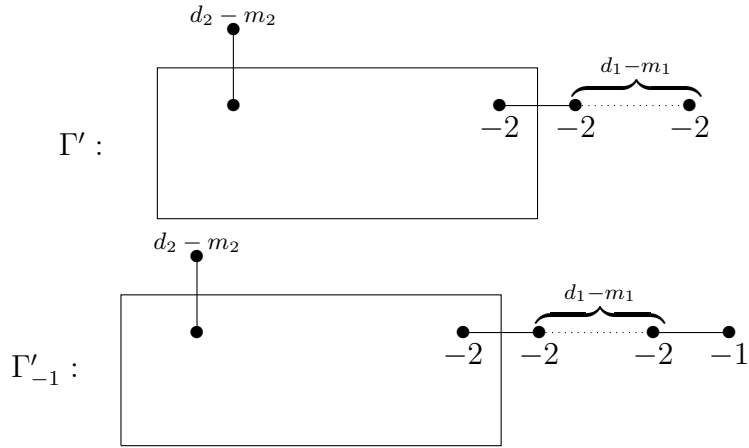
\begin{figure}[ht!]
\begin{tikzpicture}
\draw (0,0)--(0,2)--(5,2)--(5,0)--(0,0);
\draw (1,1.5) node {$\bullet$};
\draw (1,2.5) node {$\bullet$};
\draw (1,1.5)--(1,2.5);
\draw (4.5,1.5) node {$\bullet$};
\draw (5.5, 1.5) node {$\bullet$};
\draw (7, 1.5) node {$\bullet$};
\draw (4.5,1.5)--(5.5,1.5);
\draw [dotted] (5.5,1.5)--(7,1.5);
\draw (-1,1) node {$\Gamma':$};
\draw (4.5,1.2) node {$-2$};
\draw (1,2.7) node {\scriptsize{$d_2-m_2$}};
\draw (5.5,1.2) node {$-2$};
\draw (7,1.2) node {$-2$};
\draw (6.3,1.9)  node {$\overbrace{\qquad\qquad}^{d_1-m_1}$};
\end{tikzpicture}
\begin{tikzpicture}
\draw (0,0)--(0,2)--(5,2)--(5,0)--(0,0);
\draw (1,1.5) node {$\bullet$};
\draw (1,2.5) node {$\bullet$};
\draw (1,1.5)--(1,2.5);
\draw (-1,1) node {$\Gamma'_{-1}:$};
\draw (4.5,1.5) node {$\bullet$};
\draw (5.5, 1.5) node {$\bullet$};
\draw (7, 1.5) node {$\bullet$};
\draw (4.5,1.5)--(5.5,1.5);
\draw [dotted] (5.5,1.5)--(7,1.5);
\draw (4.5,1.2) node {$-2$};
\draw (1,2.7) node {\scriptsize{$d_2-m_2$}};
\draw (5.5,1.2) node {$-2$};
\draw (7,1.2) node {$-2$};
\draw (6.3,1.9)  node {$\overbrace{\qquad\qquad}^{d_1-m_1}$};
\draw (8,1.5) node {$\bullet$};
\draw (8,1.2) node {$-1$};
\draw (7,1.5)--(8,1.5);
\end{tikzpicture}
\label{fig: unknotted}
\caption{The graphs $\Gamma$ (top) and $\Gamma'$ (bottom)}
\end{figure}

Let us add an extra $(-1)$--vertex to the end of this chain, and call the graph with this extra vertex $\Gamma'_{-1}$. By consecutively blowing down  this $(-1)$--vertex and  the chain of $(-2)$-vertices, we obtain a graph representing $S^3_{d_2}(L_2)=L(d_2,1)$. For $d_2\ll 0$ we conclude that
$\Gamma'_{-1}$ is negative definite and rational.
But $\Gamma'$ is a subgraph of $\Gamma'_{-1}$, so it is rational too.

If $d_1=m_1$ then by plumbing calculus one can delete the zero--vertex and its support vertex $v_1$,
hence we need to show that the surgery along $v_2$ of $\Gamma\setminus v_1$ is rational.
But $v_2$ is simple in $\Gamma$ (by $(4)\Leftrightarrow(1)$), hence it is simple in $\Gamma\setminus v_1$ too.
Hence we can conclude by Theorem \ref{th: l space simple vertex}.
\end{proof}

\begin{corollary}
If $L$  is an algebraic link with two components and (at least) one of the components is unknot
then $\ls(L)$ is not bounded from below.
\end{corollary}
\begin{remark}
By the classification of algebraic knots
(or by Proposition \ref{alex 1 component}) $L_2$ is unknot if and only if its Alexander polynomial
$\Delta_2(t)=1$. Hence, by formula of Torres \cite{Torres}, this happens exactly when
$\Delta(1,t)=(1-t^l)/(1-t)$, where $l$ is the linking number.
\end{remark}

\section{When is $\ls(L)$ bounded from below?}\label{s:PROOFS}

\subsection{}
In this section we provide several characterizations of the boundedness from below of
$\ls(L)$ for algebraic links (in particular, we prove Theorem \ref{th:alg alex}).

Let $\Gamma$ be the dual graph of the minimal good embedded resolution of $C$ \cite{EN}.
If the strict transforms of $C_1$ and $C_2$ intersect the same irreducible exceptional component, say $E_v$,
then we call $C_1$ and $C_2$
(and $L_1$ and $L_2$)  {\em parallel}.
Otherwise the strict transforms of $C_1$ and $C_2$  intersect transversally two different components,
let their index be $v_1$ and $v_2$.

\begin{theorem}
\label{th:alg alex2}
For  a 2--component  algebraic link $L$  the following facts are equivalent:
\begin{enumerate}
\item the intersections of $\ls(L)$ with lines $\{m_1\}\times \BZ$ and $\BZ\times \{m_2\}$ are both bounded from below;
\item $\ls(L)$ is bounded from below;
\item there exists a very good point $v\in\BZ^2$ for $L$;  
\item  $\Delta(L)$ is not of ordered type;
\item $L$ is not parallel, and the vertex $v_1$ is not simple in the graph $\Gamma\setminus v_2$, and
 $v_2$ is not simple in the graph $\Gamma\setminus v_1$
(in the sense of Definition \ref{def simple}).
\end{enumerate}
\end{theorem}

\subsection{}\label{ss:proofofalg}
First we outline its proof. 
Part $(2)\Rightarrow (1)$ is clear,  $(3)\Rightarrow (2)$ follows from Theorem \ref{th:bounded},
$(4)\Rightarrow (3)$ is proven in Lemma \ref{alg very good}. Part
$(1)\Rightarrow (5)$ will follow from the next Lemma.

\begin{lemma}
\label{th:simple to unbounded}
a) If\, $C_1$ and $C_2$ are parallel then  $\Delta $ is of ordered type, and $m_1=m_2$, and
$((\{m_1\}\times \BZ)\cup (\BZ\times \{m_2\}))\setminus (m_1,m_2)\subset \ls(L)$.

b) If $v_2$ is a simple vertex  for $\Gamma\setminus v_1$ then $\ls(L)\cap (\{m_1\}\times \BZ)$ is unbounded from below.
\end{lemma}

In fact, the previous implications together with Lemma \ref{th:simple to unbounded}
finish completely the case of parallel components.
Finally, it remains to prove $(5)\Rightarrow (4)$ in the non--parallel case.


\begin{proof}[Proof of  Lemma \ref{th:simple to unbounded}] (a)
Consider the  surgery manifold $S^3_{m_1,d_2}(L)$. It is represented
by a graph $\Gamma$ with two additional vertices with decorations $0$ and $d_2-m_2$, connected to $v$.
By 0--splitting (cf. \cite{Neumann}) the first new vertex together with $v$ can be deleted, and we
remain with the rational graph $\Gamma\setminus v$ and another component consisting of
 a single vertex decorated by $d_2-m_2$, which is a lens space whenever $d_2\not=m_2$.
Furthermore, by \eqref{eq:EN},  ${\rm Supp}(\Delta)$ sits  an a line.

(b) Consider again $S^3_{m_1,d_2}(L)$. By 0--splitting (of the first new vertex)
this is equivalent with the surgery  of $\Gamma\setminus v_1$ along $v_2$.
But this is rational for $d_2\ll 0$ by Theorem \ref{th: l space simple vertex}.
\end{proof}
\subsection{} The remained implication $(5)\Rightarrow (4)$ (for non--parallel case) will be proved in two steps.
We need to prove that whenever $v_1\not= v_2$ and if $v_i$ is not simple in $\Gamma\setminus v_j$
($(i,j)=(1,2)$ or $(2,1)$) then $\Delta $ is not ordered.
First we consider that particular family of graphs when both $v_1$ and $v_2$ are
$(-1)$--vertices. In this case the assumption is satisfied. Indeed, the two $(-1)$--vertices cannot be
adjacent (since $\Gamma$ is negative definite), hence $v_1$ has at least two adjacent vertices
in $\Gamma\setminus v_2$. Since any multiplicity of the minimal cycle is at least one, a $(-1)$--vertex with
at least two neighbors cannot have multiplicity one in the minimal cycle (cannot be in the Lipman cone).

The first step of $(5)\Rightarrow (4)$ is the following.

\begin{theorem} \label{prop:splice1}
Assume that the two components of $L$ are not parallel, and both arrowheads are supported by $(-1)$--vertices.
Then $\Delta$ is not of ordered type.
\end{theorem}

\begin{proof}[Proof of Theorem \ref{prop:splice1}]

It is convenient to use the following terminology, which helps to test
Alexander polynomials of non--ordered type. A polynomial $\Delta\in \BZ[t_1,t_2]$ is an extension of
$\Delta'\in\BZ[t_1,t_2]$, if there exist polynomials $P_1,\ldots,P_N\in\BZ[t_1,t_2]$,
such that $\Delta=P_1\cdots P_N\cdot \Delta'$,  and all the non--zero coefficients of all the $P_i$'s
and $\Delta'$  are positive. We call the polynomials $P_i$ {\em extension factors}. Since any coefficients of
$\Delta$ is 0 or 1, any non--zero monomial of $P_1\cdots P_N$ gives a shifted copy
  of ${\rm Supp}(\Delta')$ is ${\rm Supp}(\Delta)$.
  In particular, if $\Delta'$ is not of ordered type,
  then $\Delta$ necessarily is not of ordered type as well.

We need to discuss two families of splice diagrams following \cite[App. to Ch. I]{EN}.
Recall (cf. \cite{EN}) that $m_{uv}$ (needed in the formula of $\Delta$)
reads from the diagram as the product of decorations along but
not on the path connecting $v_i$ and $v_j$.

(I) \ The first one has the following form with $s> n$ and $r> n$
(these inequalities imply that the supporting vertices are automatically $(-1)$--vertices in $\Gamma$):

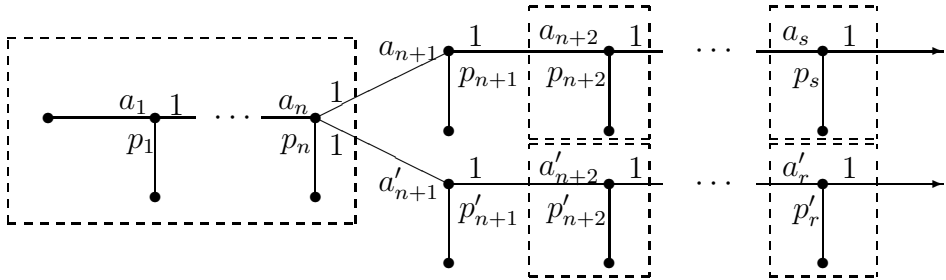
\begin{picture}(400,130)(20,-10)
\put(60,60){\circle*{4}}
\put(100,60){\circle*{4}}
\put(100,30){\circle*{4}}
\put(92,65){\makebox(0,0){$a_1$}}
\put(108,65){\makebox(0,0){$1$}}
\put(95,50){\makebox(0,0){$p_1$}}
\put(100,30){\line(0,1){30}}
\put(60,60){\line(1,0){55}}
\put(130,60){\makebox(0,0){$\cdots$}}

\put(160,60){\circle*{4}}
\put(160,30){\circle*{4}}
\put(160,30){\line(0,1){30}}
\put(152,65){\makebox(0,0){$a_n$}}
\put(168,70){\makebox(0,0){$1$}}
\put(168,50){\makebox(0,0){$1$}}
\put(153,50){\makebox(0,0){$p_n$}}
\put(160,60){\line(-1,0){20}}

\put(160,60){\line(2,1){50}}
\put(160,60){\line(2,-1){50}}

\put(210,85){\circle*{4}}
\put(210,35){\circle*{4}}
\put(210,85){\line(0,-1){30}}
\put(210,35){\line(0,-1){30}}
\put(210,55){\circle*{4}}
\put(210,5){\circle*{4}}
\put(195,85){\makebox(0,0){$a_{n+1}$}}
\put(195,35){\makebox(0,0){$a'_{n+1}$}}
\put(210,85){\line(1,0){80}}
\put(210,35){\line(1,0){80}}
\put(225,75){\makebox(0,0){$p_{n+1}$}}
\put(225,25){\makebox(0,0){$p'_{n+1}$}}
\put(220,91){\makebox(0,0){$1$}}
\put(220,41){\makebox(0,0){$1$}}
\put(310,85){\makebox(0,0){$\cdots$}}
\put(310,35){\makebox(0,0){$\cdots$}}

\put(350,85){\circle*{4}}
\put(350,35){\circle*{4}}
\put(350,55){\circle*{4}}
\put(350,5){\circle*{4}}
\put(350,85){\line(0,-1){30}}
\put(350,35){\line(0,-1){30}}
\put(325,85){\vector(1,0){70}}
\put(325,35){\vector(1,0){70}}
\put(360,91){\makebox(0,0){$1$}}
\put(360,41){\makebox(0,0){$1$}}
\put(340,91){\makebox(0,0){$a_s$}}
\put(340,41){\makebox(0,0){$a'_r$}}
\put(344,75){\makebox(0,0){$p_{s}$}}
\put(344,25){\makebox(0,0){$p'_{r}$}}

\put(270,85){\circle*{4}}
\put(270,35){\circle*{4}}
\put(270,55){\circle*{4}}
\put(270,5){\circle*{4}}
\put(270,85){\line(0,-1){30}}
\put(270,35){\line(0,-1){30}}

\put(280,91){\makebox(0,0){$1$}}
\put(280,41){\makebox(0,0){$1$}}
\put(255,91){\makebox(0,0){$a_{n+2}$}}
\put(255,41){\makebox(0,0){$a'_{n+2}$}}
\put(258,75){\makebox(0,0){$p_{n+2}$}}
\put(258,25){\makebox(0,0){$p'_{n+2}$}}

\put(45,20){\dashbox{3}(130,70)}
\put(240,52){\dashbox{3}(45,50)}
\put(240,0){\dashbox{3}(45,50)}
\put(330,52){\dashbox{3}(40,50)}
\put(330,0){\dashbox{3}(40,50)}
\end{picture}

We decompose the set of vertices in several disjoint subsets, accordingly $\Delta$ will be a product of polynomials. The contribution of the
vertices from the left dash--box is $P_0(t_1,t_2)=\widetilde{P}_0
(t_1^{p_{n+1}\cdots p_s}\cdot t_2^{p_{n+1}'\cdots p_r'})$, where
$$\widetilde{P}_0(t)=\frac{ (t^{a_1p_1\cdots p_n}-1)(t^{a_2p_2\cdots p_n}-1)\cdots (t^{a_np_n}-1)^2}
{ (t^{p_1\cdots p_n}-1)(t^{a_1p_2\cdots p_n}-1)\cdots (t^{a_n}-1)}.$$
 Note that $\widetilde{P}_0(t_1t_2)$ is the Alexander polynomial of the link (with parallel components),
 determined by the
 diagram in the dash--box (and its two arrows correspond  to the cutting edges).  Hence, by
 (\ref{eq:av}), all the non--zero coefficients of $\widetilde{P}$ are 1, and  $P_0$ can be an
 extension factor.

 The contributions from small dash--boxes are also extension factors. Indeed, the box containing
 the vertex with adjacent weight $ a_{n+2},\, p_{n+2}$ and 1 has the multiplicative contribution
 $P_{n+2}(t_1,t_2)=\widetilde{P}_{n+2}(t_1^{a_{n+2}p_{n+3}\cdots p_s}\cdot t_2^{p_{n+1}a_np_np_{n+1}'
\cdots p_r'})$, where $\widetilde{P}_{n+2}(t)=(t^{p_{n+2}}-1)/(t-1)$. We will denote these
extension factors by $P_{n+2},\ldots, P_s$ and $P_{n+2}',\ldots, P_r'$.

The contribution from the remaining  four vertices is
 $\Delta'(t_1,t_2)=\widetilde{\Delta}'
(t_1^{p_{n+2}\cdots p_s}, t_2^{p_{n+2}'\cdots p_r'})$,  where
$$\widetilde{\Delta}'(t_1,t_2)=\frac{(t_1^{p_{n+1}a_{n+1}} t_2^{p_{n+1}a_np_np_{n+1}'}-1)}
 {(t_1^{a_{n+1}} t_2^{a_np_np_{n+1}'}-1)}\cdot
 \frac{( t_1^{p_{n+1}'p_na_np_{n+1}}t_2^{a_{n+1}'p_{n+1}'}-1)}
 {(t_1^{p_na_np_{n+1}}t_2^{a_{n+1}'} -1)}.
$$
Obviously,  $\Delta'$ is not of ordered type if and only if
$\widetilde{\Delta}'$ is not of ordered type.
But $\widetilde{\Delta}'$ is not of ordered type since in its support one has the
two lattice points $(a_{n+1}, a_np_np_{n+1}')$ and $(p_na_np_{n+1}, a_{n+1}')$.
We recall (see \cite[page 51]{EN}) that for an irreducible component
 the $(p_i,a_i)_i$ splice diagram decorations
are related with the Newton pairs  by $a_{i+1}=q_{i+1}+p_ip_{i+1}a_i$.
Therefore,  $a_{n+1}-p_na_np_{n+1}=q_{n+1}>0$ and
$a_{n+1}'-a_np_np_{n+1}'=q_{n+1}'>0$. (These inequalities follow from Lemma \ref{BraunN}(a) as well.)

Since $\Delta=P_0\cdot P_{n+2}\cdots P_r\cdot P_{n+2}'\cdots P_r'\cdot \Delta'$, the polynomial $\Delta$ is not of ordered type. \\

(I.Deg) \ Let us show that the `degenerate cases'  of  family (I), when
$s=n$ and/or $r=n$ cannot occur. Indeed, $s=n$ and $r=n$ cannot happen, since this is exactly the parallel case.
If $s>n$ and $r=n$, then one of the supporting vertices (say $v_2$)
 is the $n$-th node, which is not the last node of  $C_1$.
This $v_2$ cannot be a $(-1)$--vertex (in fact, its decoration is $-1-\lceil p_{n+1}/q_{n+1}\rceil$, where
$(p_i,q_i)$ are the Newton pairs of $C_1$, cf. \cite{N2}).
Hence this case cannot occur as well. \\

(II) \ The second family  has the next  form, again with $s>n$ and $r> n$.
Like above, these inequalities guarantee that both supporting vertices are automatically $(-1)$--vertices.

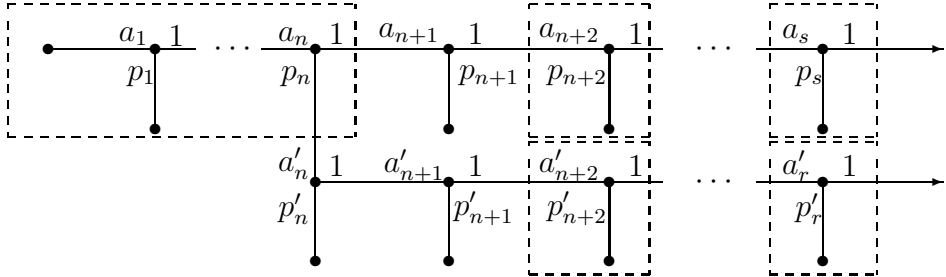
\begin{picture}(400,130)(20,-10)
\put(60,85){\circle*{4}}
\put(100,85){\circle*{4}}
\put(100,55){\circle*{4}}
\put(92,90){\makebox(0,0){$a_1$}}
\put(108,90){\makebox(0,0){$1$}}
\put(95,75){\makebox(0,0){$p_1$}}
\put(100,55){\line(0,1){30}}
\put(60,85){\line(1,0){55}}
\put(130,85){\makebox(0,0){$\cdots$}}

\put(160,85){\circle*{4}}

\put(152,90){\makebox(0,0){$a_n$}}
\put(168,91){\makebox(0,0){$1$}}
\put(153,75){\makebox(0,0){$p_n$}}
\put(160,85){\line(-1,0){20}}

\put(160,85){\line(1,0){50}}
\put(160,85){\line(0,-1){50}}
\put(160,35){\circle*{4}}
\put(160,35){\line(1,0){50}}
\put(160,35){\line(0,-1){30}}
\put(160,5){\circle*{4}}
\put(152,42){\makebox(0,0){$a'_n$}}
\put(152,25){\makebox(0,0){$p'_n$}}
\put(168,41){\makebox(0,0){$1$}}

\put(210,85){\circle*{4}}
\put(210,35){\circle*{4}}
\put(210,85){\line(0,-1){30}}
\put(210,35){\line(0,-1){30}}
\put(210,55){\circle*{4}}
\put(210,5){\circle*{4}}
\put(195,91){\makebox(0,0){$a_{n+1}$}}
\put(197,41){\makebox(0,0){$a'_{n+1}$}}
\put(210,85){\line(1,0){80}}
\put(210,35){\line(1,0){80}}
\put(225,75){\makebox(0,0){$p_{n+1}$}}
\put(223,25){\makebox(0,0){$p'_{n+1}$}}
\put(220,91){\makebox(0,0){$1$}}
\put(220,41){\makebox(0,0){$1$}}
\put(310,85){\makebox(0,0){$\cdots$}}
\put(310,35){\makebox(0,0){$\cdots$}}

\put(350,85){\circle*{4}}
\put(350,35){\circle*{4}}
\put(350,55){\circle*{4}}
\put(350,5){\circle*{4}}
\put(350,85){\line(0,-1){30}}
\put(350,35){\line(0,-1){30}}
\put(325,85){\vector(1,0){70}}
\put(325,35){\vector(1,0){70}}
\put(360,91){\makebox(0,0){$1$}}
\put(360,41){\makebox(0,0){$1$}}
\put(340,91){\makebox(0,0){$a_s$}}
\put(340,41){\makebox(0,0){$a'_r$}}
\put(345,75){\makebox(0,0){$p_{s}$}}
\put(345,25){\makebox(0,0){$p'_{r}$}}

\put(270,85){\circle*{4}}
\put(270,35){\circle*{4}}
\put(270,55){\circle*{4}}
\put(270,5){\circle*{4}}
\put(270,85){\line(0,-1){30}}
\put(270,35){\line(0,-1){30}}

\put(280,91){\makebox(0,0){$1$}}
\put(280,41){\makebox(0,0){$1$}}
\put(255,91){\makebox(0,0){$a_{n+2}$}}
\put(255,41){\makebox(0,0){$a'_{n+2}$}}
\put(258,75){\makebox(0,0){$p_{n+2}$}}
\put(258,25){\makebox(0,0){$p'_{n+2}$}}

\put(45,52){\dashbox{3}(130,50)}
\put(240,52){\dashbox{3}(45,50)}
\put(240,0){\dashbox{3}(45,50)}
\put(330,52){\dashbox{3}(40,50)}
\put(330,0){\dashbox{3}(40,50)}
\end{picture}

The contributions from the dash--boxes are extension factors as above. The contributions from the remaining three pairs of vertices is
$\Delta'(t_1,t_2)=\widetilde{\Delta}'
(t_1^{p_{n+2}\cdots p_s}, t_2^{p_{n+2}'\cdots p_r'})$, where
$$
\widetilde{\Delta}'(t_1,t_2)=
\frac{\bt^{p_{n+1}X}-1}{\bt^X-1}\cdot
\frac{\bt^{p'_{n}Z}-1}{\bt^Z-1}\cdot
\frac{\bt^{p'_{n+1}Y}-1}{\bt^Y-1} \ \  \ \ (\bt^{(a_1,a_2)}=t_1^{a_1}t_2^{a_2})$$
$$X=(a_{n+1}, a_np_n'p_{n+1}'), \ \
Z= (a_np_{n+1}, a_n'p_{n+1}'), \ \
Y=(p_n'a_n p_{n+1}, a_{n+1}').$$
Note that by edge--inequalities of the splice diagram
$p_na_n'> a_np_n'$, $a_{n+1}>a_np_np_{n+1}$, and $a'_{n+1}>a'_np'_np'_{n+1}$.
Using these,  if $a_n\leq a_n'$ then one verifies that
$Y$ and $X+(p_n'- \min\{p_n,p_n'\})Z$ is an un--ordered pair in the support of
$\widetilde{\Delta}'$.

Assume next that
 $a_n>a_n'$. Then we will use the contribution from the extension
factor $P_0(t_1,t_2)$ from the left dash--box as well.
Let $g$ be an irreducible singularity with
splice decorations $(a_i,p_i)_{i=1}^{n-1}$,
let $\Delta_g(t)$ denote the Alexander polynomial of $g$.
One can check that $P_0(t_1,t_2)=\widetilde{P}_0(t_1^{p_{n+2}\cdots p_s}, t_2^{p_{n+2}'\cdots p_r'})$, where
$$
\widetilde{P}_0(t_1,t_2)=\Delta_g(t)\frac{t^{a_n}-1}{t-1}
$$
with the substitution
$t=t_1^{p_np_{n+1}}t_2^{p_n'p_{n+1}'}$. Set also
$T:=(p_np_{n+1}, p_n'p_{n+1}')$.
One shows  (by induction the number of Newton pairs) that
$\mu(g)<a_{n-1}p_{n-1}$.
On the other hand,
 $a_{n-1}p_{n-1}< a_n/p_n< a_n'/p_n'< a_n'<a_n$, hence by Corollary \ref{cor:zeta with a factor} $\widetilde{P}_0$
has all coefficients $0$ or $1$ and the coefficient of $t^{a'_{n}}$ equals 1.
All this shows that  $a_n'\cdot T$ is in the support of $\widetilde{P}_0
\cdot \widetilde{\Delta}'$.
Then, one verifies that $Y$ and $a_n'\cdot T$ are unordered pairs in the support
of $\widetilde{P}_0 \cdot \widetilde{\Delta}'$.\\

(II.Deg) \ Next we discuss three degenerate cases of  (II) corresponding to $r=n$ and/or $s=n$.  

This is the place of a very important warning. If $r>n$ then the supporting vertex of this component
(the $r$-th node) is automatically a $(-1)$--vertex, however this is not the case
when $r=n$. Thus, if $r=n$, we have to impose this extra condition.
 The point is that if we consider the `long graph case' with $r>n$ and we wish to make induction
by considering its shorter subgraph by deleting say one splice component, the shorter graph
might not have this extra condition (hence their Alexander polynomial might be ordered). In particular, inductions
of this type cannot be implemented. The non-ordered property of $\Delta$ for long graphs ($r>n$) is imposed by the
contribution from the
long hands, for short graph ($r=n$) by the extra assumption about the existence of $(-1)$--vertices.
This explain also why we didn't handle in cases (I)-(II)  the $(n+1)$-th nodes as extension factors (though they are,
but $\Delta'$ associated with a shorter graph might not have the non--ordered property without extra assumptions).

It is not easy to combine the decoration of the splice diagram (which gives  naturally $\Delta$)
with the extra assumption regarding the $(-1)$--vertices in $\Gamma$. This is exactly the role of
Lemma \ref{lem: inequalities for z and w}. The cases (a), (b) and (a)-(b)
of \eqref{eq:3cases} correspond to the three
degenerations  of (II). \\

(II.Deg.a) \ Assume that $s=n$ but $r> n$, and in the resolution graph $\Gamma$
the supporting vertex of $L_2$ is a $(-1)$--vertex (in the splice diagram this is the node with decorations
$a_n,\ p_n, \ 1$).

Set $T:=(p_n, p_n'p_{n+1}')$,
$Z= (a_n, a_n'p_{n+1}')$ and
$Y=(p_n'a_n , a_{n+1}')$.
Then $\Delta=P_0P_{n+2}'\cdots P_r'\Delta'$, where
$P_{n+1}',\ldots, P_r'$ are extension factors, $P_0=\widetilde{P}_0(t_1,t_2^{p_{n+2}'\cdots p_r'})$,
where $\widetilde{P}_0$ is obtained from $\Delta_g(t)(t^{a_n}-1)/(t-1)$ by substitution
$t =\bt ^T$, and  $\Delta'=\widetilde{\Delta}'(t_1,t_2^{p_{n+2}'\cdots p_r'})$, where
$$
\widetilde{\Delta}'(t_1,t_2)=
\frac{\bt^{p'_{n}Z}-1}{\bt^Z-1}\cdot
\frac{\bt^{p'_{n+1}Y}-1}{\bt^Y-1}.$$
By \eqref{eq:3cases}(a)   there exists a pair of positive integers $(z,w)$ such that
$a_n'/p_n'>z/w>a_n/p_n$ and  $z<a_n$.
Note that $z>a_nw/p_n\geq a_n/p_n$ and
(by edge inequality of the diagram) $a_n'>a_np_n'/p_n\geq a_n/p_n$. But $a_n/p_n>a_{n-1}p_{n-1}>\mu(g)$.
Hence $zT\in {\rm Supp}(\widetilde{P}\widetilde{\Delta}')$ (since $z<a_n$), and
$a_n'T\in {\rm Supp}(\widetilde{P}\widetilde{\Delta}')$ provided that $a_n'<a_n$.

If $a_n'<a_n$ then $a_n'T$ and $Y$ are unordered pairs in
${\rm Supp}(\widetilde{P}\widetilde{\Delta}')$,
if 
$w<p_n'$ then $zT$ and $wZ$ are unordered. Finally, if
$a_n\leq a_n'$ and $w\geq p_n'$ then $zT$ and $Y$ are unordered pairs. \\

(II.Deg.b) \ Assume that $s>n$ and $r=n$, and in the resolution graph $\Gamma$
the supporting vertex of $L_1$ is a $(-1)$--vertex (in the splice diagram this is the node with decorations
$a_n',\ p_n', \ 1$).

Though the graph is not symmetric to the case (II.Deg.a), the computations and the proof is.
We write only the generators (and all the other details are left to the reader).
$T=(p_np_{n+1},p_n')$, $X=(a_{n+1},a_np_n')$, $Z=(a_np_{n+1},a_n')$. \eqref{eq:3cases}(b) provides a
pair $(z,w)$ with $a_n'/p_n'>z/w>a_n/p_n$ and $w<p_n'$. Then if $z<a_n$ then
$zT$ and $wZ$ are unordered, if $p_n<p_n'$ then $p_nZ$ and $X$ are unordered, and if
$z\geq a_n$ and $p_n\geq p_n'$ then $wZ$ and $X$ are unordered pairs. \\

(II.Deg.c) Assume $s=r=n$ and assume also that in the resolution graph $\Gamma$
both arrowhead--supporting vertices have decoration $(-1)$.

Set $T=(p_n,p_n')$ and $Z=(a_n,a_n')$. Then $\Delta=P_0\Delta'$, where
$P_0=\Delta_g(\bt^T)(\bt^{a_nT}-1)/(\bt^T-1)$ and $\Delta'=(\bt^{p_n'Z}-1)/(\bt^Z-1)$.
By (a)-(b) of \eqref{eq:3cases} there exist $(z,w)$ with
$a_n'/p_n'>z/w>a_n/p_n$, $z<a_n$ and $w<p_n'$. Hence $zT$ and $wZ$ are unordered.
\end{proof}

\begin{example}
\label{transversal}
Consider the algebraic link $L$ corresponding to a singularity $C=C_1\cup C_2$.
Suppose that $C_1$ and $C_2$ are singular and the tangent lines to $C_1$ and to $C_2$ are distinct.
Then the Alexander polynomial is not of ordered type (and $\ls(L)$ is bounded below, cf. Theorem
\ref{th:alg alex2}).
Indeed, $C_1$ and $C_2$ are non--parallel and both are supported by   $(-1)$--vertices.
\end{example}

\subsection{The second step of the implication $(5)\Rightarrow(4)$.} \label{ss:2step}
In order to finish the proof of Theorem \ref{th:alg alex2} we need to finish the implication
$(5)\Rightarrow(4)$ (for non--parallel case).

Consider the minimal embedded resolution of $C$. In the non--parallel case
 $C_1$ and $C_2$ are supported on different vertices $v_1$ and $v_2$.
For the case when these are both  $(-1)$--vertices, Theorem \ref{prop:splice1} states that
$\Delta(t_1,t_2)$ is not of ordered type. Hence, we need to consider the case when only one component (say, $C_2$) is supported at the $(-1)$--vertex, and  $C_1$ is resolved automatically by the minimal resolution of $C_2$.


\begin{theorem}\label{th:ordered to simple}
Assume that the resolution graph $\Gamma$ is the minimal good resolution of $C_2$, and
the arrowhead of $C_1$ is glued to some arbitrarily chosen vertex of $\Gamma\setminus v_2$.
If $\Delta(t_1,t_2)$ is of ordered type then $v_1$ is a simple vertex in $\Gamma\setminus v_2$.
\end{theorem}
\begin{proof}  
Note that $\Gamma\setminus v_2$ has two components. One of them is a string (it supports the decoration
$p_s$ in the diagram below). Since the minimal cycle of a string is reduced, if $v_1$ is
one of its vertices it automatically satisfies the wished simplicity. Thus we assume that $v_1$ is in the other component, denoted by $\Gamma'$.
We will use the following splice diagram (which codes the needed multiplicities/determinants),
where we also distinguish the vertex $v_3$,
vertex of $\Gamma'$ adjacent to $v_2$ in the resolution graph $\Gamma$. (In particular,
the splice diagram is not minimal, the $v_3$--node has valency two).
In this diagram $v_1$ is sitting between two nodes, but for any other choice of $v_1$ the proof runs identically
(see below). Set also $g:=p_{n+1}\cdots p_{s-1}$.

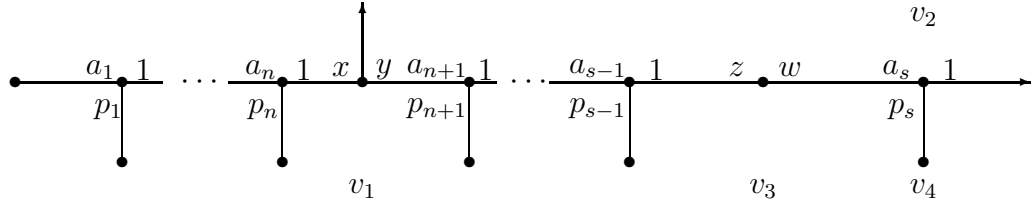
\begin{picture}(400,100)(20,0)
\put(60,60){\circle*{4}}
\put(100,60){\circle*{4}}
\put(100,30){\circle*{4}}
\put(92,65){\makebox(0,0){$a_1$}}
\put(108,65){\makebox(0,0){$1$}}
\put(95,50){\makebox(0,0){$p_1$}}
\put(100,30){\line(0,1){30}}
\put(60,60){\line(1,0){55}}
\put(130,60){\makebox(0,0){$\cdots$}}

\put(160,60){\circle*{4}}
\put(160,30){\circle*{4}}
\put(160,30){\line(0,1){30}}
\put(152,65){\makebox(0,0){$a_n$}}
\put(168,65){\makebox(0,0){$1$}}
\put(153,50){\makebox(0,0){$p_n$}}
\put(140,60){\line(1,0){100}}
\put(230,60){\circle*{4}}
\put(290,60){\circle*{4}}
\put(290,30){\circle*{4}}
\put(290,30){\line(0,1){30}}
\put(182,65){\makebox(0,0){$x$}}
\put(198,65){\makebox(0,0){$y$}}
\put(190,60){\circle*{4}}
\put(190,60){\vector(0,1){30}}
\put(400,60){\circle*{4}}
\put(260,60){\line(1,0){70}}
\put(400,30){\circle*{4}}
\put(340,60){\circle*{4}}
\put(400,30){\line(0,1){30}}
\put(230,30){\circle*{4}}
\put(230,30){\line(0,1){30}}
\put(320,60){\vector(1,0){120}}
\put(330,65){\makebox(0,0){$z$}}
\put(350,65){\makebox(0,0){$w$}}
\put(236,65){\makebox(0,0){$1$}}
\put(410,65){\makebox(0,0){$1$}}
\put(218,65){\makebox(0,0){$a_{n+1}$}}
\put(218,50){\makebox(0,0){$p_{n+1}$}}
\put(300,65){\makebox(0,0){$1$}}
\put(278,65){\makebox(0,0){$a_{s-1}$}}
\put(278,50){\makebox(0,0){$p_{s-1}$}}
\put(390,65){\makebox(0,0){$a_{s}$}}
\put(392,50){\makebox(0,0){$p_{s}$}}
\put(253,60){\makebox(0,0){$\cdots$}}
\put(190,20){\makebox(0,0){$v_1$}}
\put(340,20){\makebox(0,0){$v_3$}}
\put(400,20){\makebox(0,0){$v_4$}}
\put(400,85){\makebox(0,0){$v_2$}}
\end{picture}

Recall that $C(\Gamma)$ denote the Lipman cone \eqref{eq:conegamma}, let $pr_{12}C(\Gamma)\subset \BZ^2$
be its projection to the coordinates of $v_1$ and $v_2$. Let $m_{uv}$ denote the entries of
$-I^{-1}$ as in \ref{ss:ordered to min}. Then, for any $v\in \Vv$ the cycles $E^*_v:=\sum_um_{uv}E_u$
belongs to $C(\Gamma)$. In fact, they generate the real cone, since $(E^*_v,E_w)=0$ if $v\not=0$ and $=-1$ for
$v=w$. Hence $(m_{uv_1},m_{uv_2})\in pr_{12}C(\Gamma)$
for any $u$.
One has the following inclusion:
\begin{equation}
\label{lem:EN}
\Supp(\Delta)\subset pr_{12}C(\Gamma). 
\end{equation}
This follows from the expansion of the right hand side of \eqref{eq:EN} as power series in $t_1,t_2$.

Our goal is to construct a cycle in  $C(\Gamma')$ (hence supported on $\Gamma'$)
with $E_{v_1}$--multiplicity 1.

First we consider the projection $pr_{12}(E^*_{v_3})=(m_{v_3v_1},m_{v_3v_2})=(xgw,zp_s)$.
The entry $xgw$ can be compared with the linking number $l=m_{v_1v_2}=xgp_s$ of the components of $C$.
Indeed, if we apply \ref{BraunN}(a) for the string (without arrowhead) staying right to $v_3$ and for $G$ the
edge adjacent to $v_2$ we obtain $w<p_s$. Therefore, $xgw\leq l-xg$. On the other hand, using again
\ref{BraunN}(a) we have $x\geq a_np_ny\geq 2$. Thus $xgw\leq l-2$.

Next, $m_{v_3v_1}\geq m_{v_1v_1}$. Indeed, $m_{v_3v_1}$ (resp. $m_{v_1v_1}$) is the multiplicity
of $C_1$ along $E_{v_3}$ (resp. $E_{v_1}$), and in the resolution process of $\Gamma$
there is a sequence of blowups whose first member creates
$E_{v_1}$ and the last one   $E_{v_3}$ (here one uses the minimality of $\Gamma$ and the fact that
there is no extra blowup imposed by $C_1$). Hence $\mu(C_1)\leq m_{v_1v_1}\leq m_{v_3v_1}= xgw$.

Therefore,  by Proposition \ref{alex 1 component}
$xgw+1$ is in the semigroup of $C_1$, and  by Lemma  \ref{alex on a line} there exists $q_2$ such that
\begin{equation}\label{eq:Delta1}
Q:=(xgw+1,q_2)\in \Supp(\Delta).
\end{equation}

Next, since $w<p_s$ (see above), by  \eqref{eq:EN} $pr_{12}(wE^*_{v_4})\in {\rm Supp}(\Delta)$ too.
Its coordinates are $(wgx,wa_s)$. This can be compared with the other support point from
\eqref{eq:Delta1}. Since $\Delta$ is of ordered type, we conclude  $N:=q_2-wa_s>0$.

Using again the determinantal property \ref{BraunN}(a) for $\Gamma$ and the edge
$(v_3v_2)$, we get $wa_s=zp_s+1$. By a computation
\begin{equation}\label{eq:Delta2}
Q+N\cdot pr_{12}(E^*_{v_3})=(N+1)pr_{12}(E^*_{v_4})+(1,0).
\end{equation}
Note that by \eqref{lem:EN} $Q\in pr_{12}C(\Gamma)$, hence there exists
a cycle $D\in C(\Gamma)$ such that $pr_{12}(D)=Q+N\cdot pr_{12}(E^*_{v_3})$.
Then, finally we define the cycle $F'$ on $L(\Gamma')$ as the restriction of
$F:=D-(N+1)E^*_{v_4}$ on $\Gamma'$.
First notice that by \eqref{eq:Delta2} the $E_{v_2}$--multiplicity of $F$ is 0, therefore
for any $E_u$ with $u\in\Vv(\Gamma')$ one has $(F,E_u)_{\Gamma}=(F',E_u)_{\Gamma'}$.
But $(F,E_u)_{\Gamma}\leq 0$ since $D\in C(\Gamma)$ and $(E^*_{v_4},E_u)=0$.
Hence $F'\in C(\Gamma')$. On the other hand, the $E_{v_1}$--multiplicity of $F'$ is 1 by
\eqref{eq:Delta2}.  This ends the proof in the case of this position of the arrowhead of $C_1$.


If the $v_1$ node coincides with the $n$--th node of the diagram, then one has to make the modification
$x=a_n$  and
$g:=p_n\cdots p_{s-1}$. If it is one the $n$--th leg then $g=a_np_{n+1}\cdots p_{s-1}$. One verifies that in any
situation $xg>1$, and the above proof runs with these modifications.

This ends the proof of Theorems \ref{th:ordered to simple} and \ref{th:alg alex2}.
\end{proof}




\section{Examples}\label{s:EXAMPLES}

\subsection{}
We illustrate the above results with explicit examples.

\begin{example}
\label{whitehead}
The $h$--function for the Whitehead link (in appropriate normalization of Alexander gradings) is shown in Figure \ref{fig:whitehead}.
The bivariate Alexander polynomial equals $\Delta(t_1,t_2)=-(1-t_1^{-1})(1-t_2^{-1})$ and $c=(0,0)$.
By \eqref{stabilization h} we get
$$h_1(v_1)=h(v_1,-1)=\max(v_1,0),\ h_2(v_2)=h(-1,v_2)=\max(v_2,0).$$
The point $v=(0,0)$ (circled  in Figure \ref{fig:whitehead}) is good ($h(0,0)=1$ while $h_1(0)=h_2(0)=0$) and self-dual, so it is very good. Therefore all L--space surgeries on the Whitehead link have positive coefficients,
in agreement with \cite[Proposition 6.4]{Liu}.
\end{example}

\begin{figure}[ht]
\begin{tikzpicture}
\draw (0,0) node {$0$};
\draw (1,0) node {$0$};
\draw (2,0) node {$1$};
\draw (3,0) node {$2$};
\draw (4,0) node {$3$};
\draw (0,1) node {$0$};
\draw (1,1) node {$1$};
\draw (2,1) node {$1$};
\draw (3,1) node {$2$};
\draw (4,1) node {$3$};
\draw (0,2) node {$1$};
\draw (1,2) node {$1$};
\draw (2,2) node {$2$};
\draw (3,2) node {$3$};
\draw (4,2) node {$4$};
\draw (0,3) node {$2$};
\draw (1,3) node {$2$};
\draw (2,3) node {$3$};
\draw (3,3) node {$4$};
\draw (4,3) node {$5$};
\draw (0,4) node {$3$};
\draw (1,4) node {$3$};
\draw (2,4) node {$4$};
\draw (3,4) node {$5$};
\draw (4,4) node {$6$};

\draw (1,1) circle(0.3);

\draw [dashed] (0.2,0)--(0.8,0);
\draw [dashed] (1.2,0)--(1.8,0);
\draw [dashed] (2.2,0)--(2.8,0);
\draw [dashed] (3.2,0)--(3.8,0);
\draw [dashed] (0.2,1)--(0.8,1);
\draw [dashed] (1.2,1)--(1.8,1);
\draw [dashed] (2.2,1)--(2.8,1);
\draw [dashed] (3.2,1)--(3.8,1);
\draw [dashed] (0.2,2)--(0.8,2);
\draw [dashed] (1.2,2)--(1.8,2);
\draw [dashed] (2.2,2)--(2.8,2);
\draw [dashed] (3.2,2)--(3.8,2);
\draw [dashed] (0.2,3)--(0.8,3);
\draw [dashed] (1.2,3)--(1.8,3);
\draw [dashed] (2.2,3)--(2.8,3);
\draw [dashed] (3.2,3)--(3.8,3);
\draw [dashed] (0.2,4)--(0.8,4);
\draw [dashed] (1.2,4)--(1.8,4);
\draw [dashed] (2.2,4)--(2.8,4);
\draw [dashed] (3.2,4)--(3.8,4);

\draw [dashed] (0,0.2)--(0,0.8);
\draw [dashed] (1,0.2)--(1,0.8);
\draw [dashed] (2,0.2)--(2,0.8);
\draw [dashed] (3,0.2)--(3,0.8);
\draw [dashed] (4,0.2)--(4,0.8);
\draw [dashed] (0,1.2)--(0,1.8);
\draw [dashed] (1,1.2)--(1,1.8);
\draw [dashed] (2,1.2)--(2,1.8);
\draw [dashed] (3,1.2)--(3,1.8);
\draw [dashed] (4,1.2)--(4,1.8);
\draw [dashed] (0,2.2)--(0,2.8);
\draw [dashed] (1,2.2)--(1,2.8);
\draw [dashed] (2,2.2)--(2,2.8);
\draw [dashed] (3,2.2)--(3,2.8);
\draw [dashed] (4,2.2)--(4,2.8);
\draw [dashed] (0,3.2)--(0,3.8);
\draw [dashed] (1,3.2)--(1,3.8);
\draw [dashed] (2,3.2)--(2,3.8);
\draw [dashed] (3,3.2)--(3,3.8);
\draw [dashed] (4,3.2)--(4,3.8);

\draw (0,-0.5) node {$\scriptstyle{-1}$};
\draw (1,-0.5) node {$\scriptstyle{0}$};
\draw (2,-0.5) node {$\scriptstyle{1}$};
\draw (3,-0.5) node {$\scriptstyle{2}$};
\draw (4,-0.5) node {$\scriptstyle{3}$};
\draw (4.5,-0.5) node {$\scriptstyle{v_1}$};

\draw (-0.5,0) node {$\scriptstyle{-1}$};
\draw (-0.5,1) node {$\scriptstyle{0}$};
\draw (-0.5,2) node {$\scriptstyle{1}$};
\draw (-0.5,3) node {$\scriptstyle{2}$};
\draw (-0.5,4) node {$\scriptstyle{3}$};
\draw (-0.5,4.5) node {$\scriptstyle{v_2}$};
\end{tikzpicture}
\caption{$h$--function for the Whitehead link}
\label{fig:whitehead}
\end{figure}
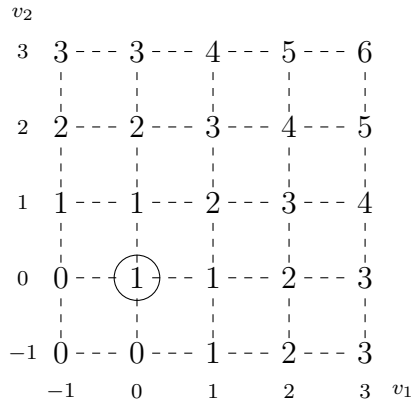

\begin{example}
\label{two cusps}
Consider the singularity $(x^2-y^3)(x^3-y^2)=0$, corresponding to a pair of trefoils with linking number 4.
The bivariate Alexander polynomial equals $\Delta(t_1,t_2)=(1+t_1^2t_2^3)(1+t_1^3t_2^2)$. The graph of the $h$--function is given in the Figure \ref{fig: Hilb two cusps} (the semigroup is bold). For a point $v=(3,3)$ one has $h(v)=3$ and $h_1(v)=h_2(v)=2$.
Since $c=(6,6)$, one has $v^*=v$, so $v$ is very good. For $\ls(L)$ see Figure \ref{fig: ls two cusps}.
\end{example}

\begin{figure}[ht]
\begin{tikzpicture}
\draw (0,0) node {$\mathbf{0}$};
\draw (1,0) node {$1$};
\draw (2,0) node {$1$};
\draw (3,0) node {$2$};
\draw (4,0) node {$3$};
\draw (5,0) node {$4$};
\draw (6,0) node {$5$};
\draw (7,0) node {$6$};
\draw (0,1) node {$1$};
\draw (1,1) node {$1$};
\draw (2,1) node {$1$};
\draw (3,1) node {$2$};
\draw (4,1) node {$3$};
\draw (5,1) node {$4$};
\draw (6,1) node {$5$};
\draw (7,1) node {$6$};
\draw (0,2) node {$1$};
\draw (1,2) node {$1$};
\draw (2,2) node {$\mathbf{1}$};
\draw (3,2) node {$\mathbf{2}$};
\draw (4,2) node {$3$};
\draw (5,2) node {$4$};
\draw (6,2) node {$5$};
\draw (7,2) node {$6$};
\draw (0,3) node {$2$};
\draw (1,3) node {$2$};
\draw (2,3) node {$\mathbf{2}$};
\draw (3,3) node {$3$};
\draw (4,3) node {$3$};
\draw (5,3) node {$4$};
\draw (6,3) node {$5$};
\draw (7,3) node {$6$};
\draw (0,4) node {$3$};
\draw (1,4) node {$3$};
\draw (2,4) node {$3$};
\draw (3,4) node {$3$};
\draw (4,4) node {$\mathbf{3}$};
\draw (5,4) node {$\mathbf{4}$};
\draw (6,4) node {$\mathbf{5}$};
\draw (7,4) node {$\mathbf{6}$};
\draw (0,5) node {$4$};
\draw (1,5) node {$4$};
\draw (2,5) node {$4$};
\draw (3,5) node {$4$};
\draw (4,5) node {$\mathbf{4}$};
\draw (5,5) node {$\mathbf{5}$};
\draw (6,5) node {$6$};
\draw (7,5) node {$7$};
\draw (0,6) node {$5$};
\draw (1,6) node {$5$};
\draw (2,6) node {$5$};
\draw (3,6) node {$5$};
\draw (4,6) node {$\mathbf{5}$};
\draw (5,6) node {$6$};
\draw (6,6) node {$\mathbf{6}$};
\draw (7,6) node {$\mathbf{7}$};
\draw (0,7) node {$6$};
\draw (1,7) node {$6$};
\draw (2,7) node {$6$};
\draw (3,7) node {$6$};
\draw (4,7) node {$\mathbf{6}$};
\draw (5,7) node {$7$};
\draw (6,7) node {$\mathbf{7}$};
\draw (7,7) node {$\mathbf{8}$};

\draw (3,3) circle(0.3);

\draw [dashed] (0.2,0)--(0.8,0);
\draw [dashed] (1.2,0)--(1.8,0);
\draw [dashed] (2.2,0)--(2.8,0);
\draw [dashed] (3.2,0)--(3.8,0);
\draw [dashed] (4.2,0)--(4.8,0);
\draw [dashed] (5.2,0)--(5.8,0);
\draw [dashed] (6.2,0)--(6.8,0);
\draw [dashed] (0.2,1)--(0.8,1);
\draw [dashed] (1.2,1)--(1.8,1);
\draw [dashed] (2.2,1)--(2.8,1);
\draw [dashed] (3.2,1)--(3.8,1);
\draw [dashed] (4.2,1)--(4.8,1);
\draw [dashed] (5.2,1)--(5.8,1);
\draw [dashed] (6.2,1)--(6.8,1);
\draw [dashed] (0.2,2)--(0.8,2);
\draw [dashed] (1.2,2)--(1.8,2);
\draw [dashed] (2.2,2)--(2.8,2);
\draw [dashed] (3.2,2)--(3.8,2);
\draw [dashed] (4.2,2)--(4.8,2);
\draw [dashed] (5.2,2)--(5.8,2);
\draw [dashed] (6.2,2)--(6.8,2);
\draw [dashed] (0.2,3)--(0.8,3);
\draw [dashed] (1.2,3)--(1.8,3);
\draw [dashed] (2.2,3)--(2.8,3);
\draw [dashed] (3.2,3)--(3.8,3);
\draw [dashed] (4.2,3)--(4.8,3);
\draw [dashed] (5.2,3)--(5.8,3);
\draw [dashed] (6.2,3)--(6.8,3);
\draw [dashed] (0.2,4)--(0.8,4);
\draw [dashed] (1.2,4)--(1.8,4);
\draw [dashed] (2.2,4)--(2.8,4);
\draw [dashed] (3.2,4)--(3.8,4);
\draw [dashed] (4.2,4)--(4.8,4);
\draw [dashed] (5.2,4)--(5.8,4);
\draw [dashed] (6.2,4)--(6.8,4);
\draw [dashed] (0.2,5)--(0.8,5);
\draw [dashed] (1.2,5)--(1.8,5);
\draw [dashed] (2.2,5)--(2.8,5);
\draw [dashed] (3.2,5)--(3.8,5);
\draw [dashed] (4.2,5)--(4.8,5);
\draw [dashed] (5.2,5)--(5.8,5);
\draw [dashed] (6.2,5)--(6.8,5);
\draw [dashed] (0.2,6)--(0.8,6);
\draw [dashed] (1.2,6)--(1.8,6);
\draw [dashed] (2.2,6)--(2.8,6);
\draw [dashed] (3.2,6)--(3.8,6);
\draw [dashed] (4.2,6)--(4.8,6);
\draw [dashed] (5.2,6)--(5.8,6);
\draw [dashed] (6.2,6)--(6.8,6);
\draw [dashed] (0.2,7)--(0.8,7);
\draw [dashed] (1.2,7)--(1.8,7);
\draw [dashed] (2.2,7)--(2.8,7);
\draw [dashed] (3.2,7)--(3.8,7);
\draw [dashed] (4.2,7)--(4.8,7);
\draw [dashed] (5.2,7)--(5.8,7);
\draw [dashed] (6.2,7)--(6.8,7);

\draw [dashed] (0,0.2)--(0,0.8);
\draw [dashed] (1,0.2)--(1,0.8);
\draw [dashed] (2,0.2)--(2,0.8);
\draw [dashed] (3,0.2)--(3,0.8);
\draw [dashed] (4,0.2)--(4,0.8);
\draw [dashed] (5,0.2)--(5,0.8);
\draw [dashed] (6,0.2)--(6,0.8);
\draw [dashed] (7,0.2)--(7,0.8);
\draw [dashed] (0,1.2)--(0,1.8);
\draw [dashed] (1,1.2)--(1,1.8);
\draw [dashed] (2,1.2)--(2,1.8);
\draw [dashed] (3,1.2)--(3,1.8);
\draw [dashed] (4,1.2)--(4,1.8);
\draw [dashed] (5,1.2)--(5,1.8);
\draw [dashed] (6,1.2)--(6,1.8);
\draw [dashed] (7,1.2)--(7,1.8);
\draw [dashed] (0,2.2)--(0,2.8);
\draw [dashed] (1,2.2)--(1,2.8);
\draw [dashed] (2,2.2)--(2,2.8);
\draw [dashed] (3,2.2)--(3,2.8);
\draw [dashed] (4,2.2)--(4,2.8);
\draw [dashed] (5,2.2)--(5,2.8);
\draw [dashed] (6,2.2)--(6,2.8);
\draw [dashed] (7,2.2)--(7,2.8);
\draw [dashed] (0,3.2)--(0,3.8);
\draw [dashed] (1,3.2)--(1,3.8);
\draw [dashed] (2,3.2)--(2,3.8);
\draw [dashed] (3,3.2)--(3,3.8);
\draw [dashed] (4,3.2)--(4,3.8);
\draw [dashed] (5,3.2)--(5,3.8);
\draw [dashed] (6,3.2)--(6,3.8);
\draw [dashed] (7,3.2)--(7,3.8);
\draw [dashed] (0,4.2)--(0,4.8);
\draw [dashed] (1,4.2)--(1,4.8);
\draw [dashed] (2,4.2)--(2,4.8);
\draw [dashed] (3,4.2)--(3,4.8);
\draw [dashed] (4,4.2)--(4,4.8);
\draw [dashed] (5,4.2)--(5,4.8);
\draw [dashed] (6,4.2)--(6,4.8);
\draw [dashed] (7,4.2)--(7,4.8);
\draw [dashed] (0,5.2)--(0,5.8);
\draw [dashed] (1,5.2)--(1,5.8);
\draw [dashed] (2,5.2)--(2,5.8);
\draw [dashed] (3,5.2)--(3,5.8);
\draw [dashed] (4,5.2)--(4,5.8);
\draw [dashed] (5,5.2)--(5,5.8);
\draw [dashed] (6,5.2)--(6,5.8);
\draw [dashed] (7,5.2)--(7,5.8);
\draw [dashed] (0,6.2)--(0,6.8);
\draw [dashed] (1,6.2)--(1,6.8);
\draw [dashed] (2,6.2)--(2,6.8);
\draw [dashed] (3,6.2)--(3,6.8);
\draw [dashed] (4,6.2)--(4,6.8);
\draw [dashed] (5,6.2)--(5,6.8);
\draw [dashed] (6,6.2)--(6,6.8);
\draw [dashed] (7,6.2)--(7,6.8);
\end{tikzpicture}
\caption{$h$--function for a pair of `transversal' trefoils}
\label{fig: Hilb two cusps}
\end{figure}
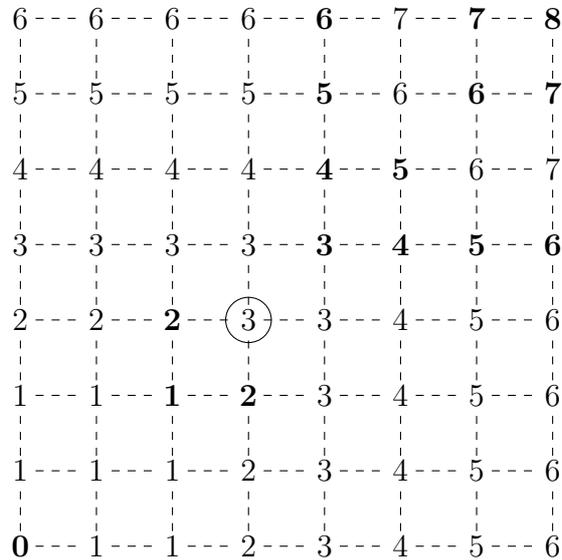

\begin{example}\label{ex:sS}
Consider the family  $L(s)$ consisting of the trefoil $L_1$ and its $(2,2s+1)$ cable $L_2$.
$L(s)$ is algebraic for $s\ge 6$:  $L_2$  is the link of the  singularity $(x,y)=(t^4,t^6+t^{2s-5})$.
Then
$$
\Delta  
=(1+t_1^4t_2^2+t_1^6t_2^3+t_1^8t_2^4+t_1^{10}t_2^5+t_1^{14}t_2^7)(1+t_1^{2s+1}t_2^6).
$$
For $s=6$ the polynomial $\Delta$  is of ordered type (see Figure \ref{fig:l6}).
For $s\ge 7$ ${\rm Supp}(\Delta)$
contains the unordered points $u=(14,7)$ and $v=(2s+1,6)$,
see Figure \ref{fig:l7} for ${\rm Supp}(\Delta(L(7)))$.

In fact, one can check (see, e.g. \cite{GN1}) that
 $(d,26)\in \ls(L(6))$ for any  $d\in \BZ$.
\end{example}

\begin{figure}[ht]
\begin{tikzpicture}
\begin{axis}[xmin=0,xmax=29,ymin=0,ymax=13,xlabel={$v_1$},ylabel={$v_2$}]
\addplot [only marks,mark=*] coordinates {
	(0,0)(4,2)(6,3)(8,4)(10,5)(14,7)
	(13,6)(17,8)(19,9)(21,10)(23,11)(27,13)
	};
\end{axis}
\end{tikzpicture}
\caption{The support of the Alexander polynomial of $L(6)$}
\label{fig:l6}
\end{figure}
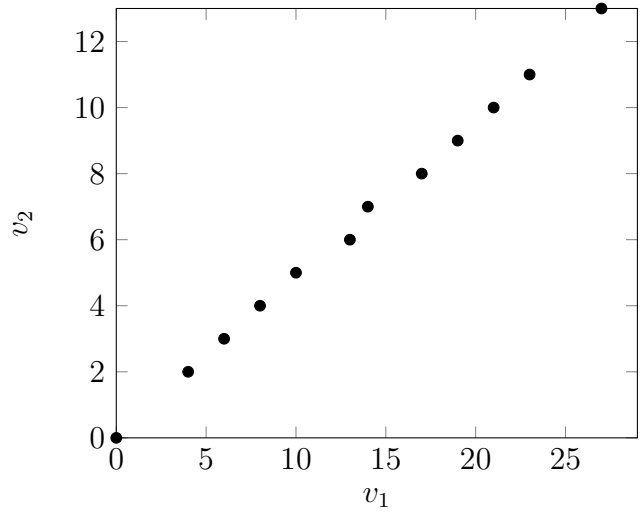

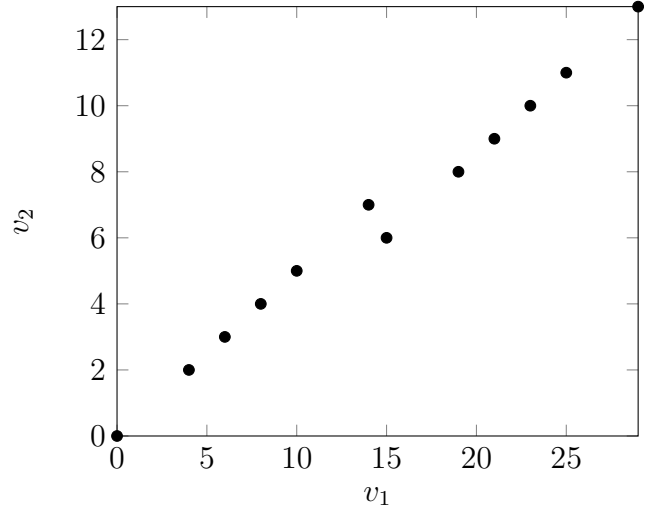
\begin{figure}[ht]
\begin{tikzpicture}
\begin{axis}[xmin=0,xmax=29,ymin=0,ymax=13,xlabel={$v_1$},ylabel={$v_2$}]
\addplot [only marks,mark=*] coordinates {
	(0,0)(4,2)(6,3)(8,4)(10,5)(14,7)
	(15,6)(19,8)(21,9)(23,10)(25,11)(29,13)
	};
\end{axis}
\end{tikzpicture}
\caption{The support of the Alexander polynomial of $L(7)$}
\label{fig:l7}
\end{figure}

\begin{example}
\label{two puiseux}
Let $L$ be an algebraic link such that $L_1$ is an $(p_1,a_1)$
torus knot, and $L_2$ is its $(p_2,a_2)$ cable, where $a_2=a_1p_1p_2+q_2$ ($q_2>0$).
(This means that $L_2$ has Newton pairs $(p_1,a_1)$, $(p_2,q_2)$, the
splice diagram  has two nodes decorated with
$(p_1,a_1)$ and $(p_2,a_2)$,   and the first node supports $L_1$.)
Then $\Delta$  is of ordered type  if and only if $q_2<p_2$.

Indeed, in this case $v_1$ is the node of the star--shaped  graph
$\Gamma\setminus v_2$ with three legs, and its
self--intersection is $e:=-\lceil p_2/q_2\rceil-1$, cf. \cite{N2}. If $q_2<p_2$ then $e\leq -3$ and $v_1$ is simple
($Z_{\min}=\sum_uE_u$),
otherwise $e=-2$ and $v_1$ is not simple in $\Gamma\setminus v_2$.
\end{example}

\begin{example}
\label{big example}
With the notation of previous example, assume that $p_1=2,a_1=3,p_2=3,q_2=2$.
This is an algebraic link consisting of the trefoil and its $(3,20)$ cable.
Since $q_2<p_2$, the Alexander polynomial of $L$ is of ordered type, cf. \ref{two puiseux},
see Figure \ref{fig:big alex} for its support.


Therefore  $\ls(L)$ is unbounded from below, nevertheless,
 only one of `test lines' $\{m_1\}\times \BZ$, $\BZ\times \{m_2\}$ confirms this fact.
The set $\ls(L)$ (computed using \cite{JProgram}) is shown in Figure \ref{fig:big ls}.
\end{example}

\begin{figure}[ht]
\begin{tikzpicture}
\begin{axis}[xmin=0,xmax=61,ymin=0,ymax=19,xlabel={$v_1$},ylabel={$v_2$}]
\addplot [only marks,mark=*] coordinates {
	(0,0)(6,2)(9,3)(12,4)(15,5)(21,7)
	(20,6)(26,8)(29,9)(32,10)(35,11)(41,13)
	(40,12)(46,14)(49,15)(52,16)(55,17)(61,19)
	};
\end{axis}
\end{tikzpicture}
\caption{The support of the Alexander polynomial in Example \ref{big example}}
\label{fig:big alex}
\end{figure}
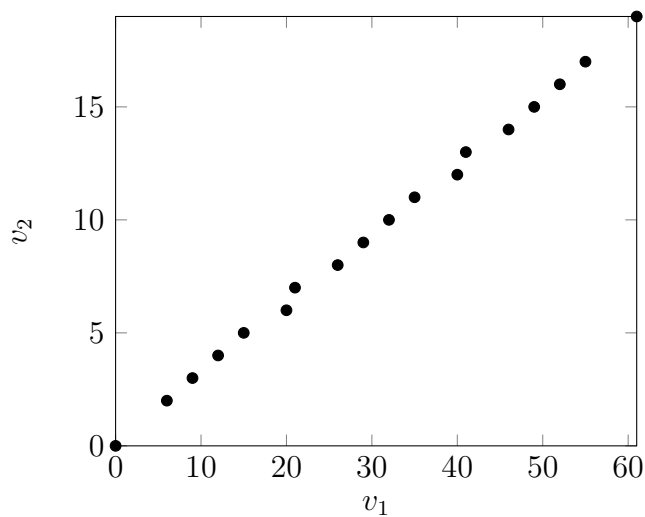

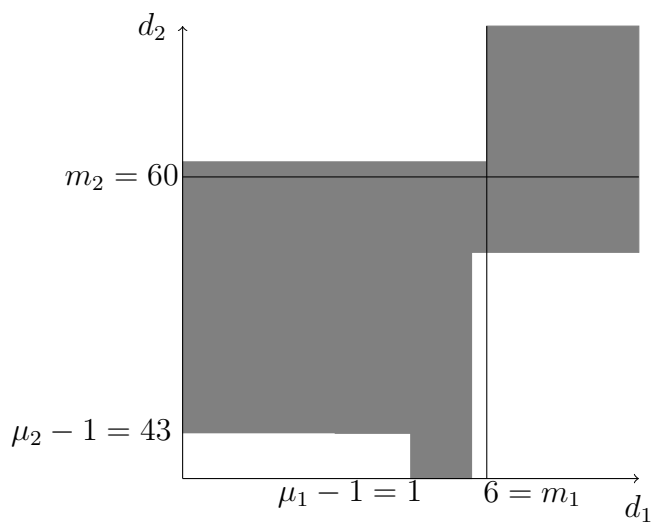
\begin{figure}[ht]
\begin{tikzpicture}[scale=0.2]
\filldraw[color=gray] (-20,-17)--(-20,1)--(0,1)--(0,10)--(10,10)--(10,-5)--(-1,-5)--(-1,-20)--(-5,-20)--(-5,-17)--(-10,-17);
\draw [->] (-20,-20)--(10,-20);
\draw [->] (-20,-20)--(-20,10);
\draw (-20,0)--(10,0);
\draw (0,-20)--(0,10);
\draw (10,-22) node {$d_1$};
\draw (-22,10) node {$d_2$};
\draw (3,-21) node {$6=m_1$};
\draw (-24,0) node {$m_2=60$};
\draw (-26,-17) node {$\mu_2-1=43$};
\draw (-9,-21) node {$\mu_1-1=1$};
\end{tikzpicture}
\caption{Sketch of $\ls(L)$ in Example \ref{big example}}
\label{fig:big ls}
\end{figure}


\begin{proof}[Proof of the statement of Example \ref{th:cables}]
Let $K$ be an L--space knot and $n/m>2g(K)-1$. Then (e.g. by \cite[Proposition 2.1]{GHom})
the $(mn,d)$--surgery on $K_{2m,2n}$ yields a connected sum
$$
S^{3}_{mn,d}(K_{2m,2n})=S^{3}_{n/m}(K)\#L(m,n)\#L(d-mn,1).
$$
Since  $n/m>2g(K)-1$, $S^{3}_{n/m}(K)$ is an L--space, so for all $d\neq mn$
the 3--manifold $S^{3}_{mn,d}(K_{2m,2n})$ is an L--space too. Therefore $(mn,d)\in \ls(K_{2m,2n})$ for all $d\neq mn$, and
$\ls(K_{2m,2n})$ is unbounded.
\end{proof}


\bigskip\bigskip

\begin{thebibliography}{99}
\bibitem{Artin62} M. Artin.
Some numerical criteria for contractibility of curves on algebraic surfaces.
Amer. J. of Math.  {\bf 84} (1962), 485--496.

\bibitem{Artin66} M. Artin.
On isolated rational singularities of surfaces.
Amer. J. of Math.  {\bf 88} (1966), 129--136.

\bibitem{BGV} S. Boyer, C. McA. Gordon, L. Watson. On L--spaces and left--orderable fundamental groups. Math. Ann. {\bf 356} (2013), no. 4, 1213--1245.

\bibitem{BR} M. Borodzik, S. Rasmussen. In preparation.

\bibitem{BraunN} G. Braun, A. N\'emethi. Surgery formula for Seiberg-Witten invariants of negative definite plumbed 3-manifolds. J. Reine Angew. Math. {\bf 638} (2010), 189--208.

\bibitem{CDG}
A. Campillo, F. Delgado, S. M. Gusein-Zade.
The Alexander polynomial of a plane curve singularity via the ring of functions on it.
Duke Math J. {\bf 117} (2003), no. 1, 125--156.

\bibitem{cdk} A. Campillo, F. Delgado, K. Kiyek. Gorenstein property and symmetry for one-dimensional local Cohen- Macaulay rings. Manuscripta Math. {\bf 83} (1994), no. 3-4, 405--423.

\bibitem{EN} D. Eisenbud, W. Neumann. Three-dimensional link theory and invariants of plane curve singularities. Annals of Mathematics Studies, 110. Princeton University Press, Princeton, NJ, 1985.

\bibitem{GHom} E. Gorsky, J. Hom. Cable links and L--space surgeries.  To appear in Quantum Topology. arXiv:1502.05425

\bibitem{GN1} E. Gorsky, A. N\'emethi. Links of plane curve singularities are L--space links.
Algebraic and Geometric Topology {\bf 16} (2016) 1905--1912.

\bibitem{GN2} E. Gorsky, A. N\'emethi. Lattice and Heegaard--Floer homologies of algebraic links.
 Int. Math. Res. Not. IMRN 2015, no. 23, 12737--12780.

\bibitem{JProgram} J. Hanselman. {\tt tree_manifolds_HFhat}, a Python program computing $\widehat{HF}$ for plumbing trees. Available at {\tt http://math.columbia.edu/$\sim$jhansel/graph_manifolds_program.html}

\bibitem{HRRW} J. Hanselman, J. Rasmussen, S. Rasmussen, L. Watson. Taut foliations on graph manifolds. arXiv:1508.05911

\bibitem{Hedden} M. Hedden. On knot Floer homology and cabling. II. Int. Math. Res. Not. IMRN 2009, no. 12, 2248--2274.

\bibitem{Kirby} R. Kirby. A calculus for framed links in $S^3$. Invent. Math. {\bf 45} (1978), no. 1, 35--56.

\bibitem{LN} T. L\'aszl\'o, A. N\'emethi.  Reduction theorem for lattice cohomology,
Int. Math. Res. Notices  {\bf 2015}, Issue 11 (2015), 2938--2985.


\bibitem{Laufer72} H.B. Laufer. On rational singularities,
Amer. J. of Math.  {\bf 94} (1972),  597--608.

\bibitem{LeTosun} D. T. L\^e, M. Tosun. Combinatorics of rational singularities. Comment. Math. Helv. {\bf 79} (2004), 582--604.

\bibitem{Lidman} T. Lidman. Framed Floer Homology. arXiv:1109.3756

\bibitem{Lipman} Lipman, J.: Rational singularities, with applications to algebraic surfaces
and unique factorization, Inst. Hautes \'Etudes Sci. Publ. Math. {\bf 36} (1969), 195-279.

\bibitem{Liu} Y. Liu. L--space surgeries on links. To appear in Quantum Topology. arXiv:1409.0075

\bibitem{MO} C. Manolescu, P. Ozsv\'ath. Heegaard Floer homology and integer surgeries on links. arXiv:1011.1317

\bibitem{Milnor} J. Milnor. Singular points of complex hypersurfaces,
Annals of Math. Studies, {\bf 61}, Princeton University Press, 1968.

\bibitem{N2} A. N\'emethi. Dedekind sums and the signature of $f(x,y)+z^N$, II. Selecta Mathematica, New Series {\bf 5} (1999), 161--179.



\bibitem{N} A. N\'emethi.  On the Ozsv\'ath-Szab\'o invariant of negative definite
plumbed 3-manifolds,
Geometry and Topology  {\bf 9} (2005), 991--1042.

\bibitem{NSurg} A. N\'emethi.  On the Heegaard Floer homology of $S^3_{-d}(K)$ and
unicuspidal rational plane curves,  {\em Fields Institute
Communications}, Vol. {\bf 47}, 2005, 219-234;
 ``Geometry and Topology of Manifolds'',
Eds: H.U. Boden, I. Hambleton, A.J. Nicas and B.D. Park,

\bibitem{NGr} A. N\'emethi. Graded roots and singularities,
(contains also the preprint `On the Heegaard Floer homology of
$S^3_{-p/q}(K)$', math.GT/0410570); Proc. {\em Advanced
School and Workshop} on {\em Singularities in Geometry and
Topology} ICTP (Trieste, Italy),
 World Sci. Publ., Hackensack, NJ, 2007, 394--463.

\bibitem{NJEMS}  A. N\'emethi. The Seiberg--Witten invariants of negative definite plumbed 3--manifolds,
Journal of EMS {\bf 13}(4) (2011), 959--974.

\bibitem{Nnew} A. N\'emethi. Links of rational singularities, L--spaces and LO fundamental groups.  arXiv:1510.07128

\bibitem{Neumann} W. Neumann. A calculus for plumbing applied to the topology of complex surface singularities and degenerating complex curves. Trans. Amer. Math. Soc. {\bf 268} (1981), no. 2, 299--344.

\bibitem{OSsurg1} P. Ozsv\'ath, Z. Szab\'o. Knot Floer homology and rational surgeries.  Algebr. Geom. Topol. {\bf 11} (2011), no. 1, 1--68.

\bibitem{OSlink} P. Ozsv\'ath, Z. Szab\'o. Holomorphic discs, link invariants and the multi-variable Alexander polynomial.
Algebr. Geom. Topol. {\bf 8} (2008), no. 2, 615--692.

\bibitem{OSsurg2}  P. Ozsv\'ath, Z. Szab\'o. On knot Floer homology and lens space surgeries. Topology {\bf 44} (2005), no. 6, 1281--1300.

\bibitem{OSplumb} P. Ozsv\'ath, Z. Szab\'o. On the Floer homology of plumbed three-manifolds. Geom. Topol. {\bf 7} (2003), 185--224.

\bibitem{OS} P. Ozsv\'ath, Z. Szab\'o.  Holomorphic disks and topological invariants for closed three-manifolds. Ann. of Math. (2) {\bf 159} (2004), no. 3, 1027--1158.

\bibitem{RR} J. Rasmussen, S. Rasmussen. Floer Simple Manifolds and L--Space Intervals. To appear in Advances in Mathematics.  arXiv:1508.05900

\bibitem{R} S. Rasmussen. L--space intervals for Graph Manifolds and Cables. To appear in Compositio Mathematica. arXiv:1511.04413

\bibitem{Sp} M. Spivakovsky.  Sandwich singularities and desingularisation of surfaces by normalized
Nash transformations, Ann. Math. 131 (1990), pp.411-491.


\bibitem{Torres} G. Torres. On the Alexander polynomial. Ann. of Math. (2) {\bf 57}, (1953). 57--89.

\bibitem{Ty} G. N. Tyurina.  Absolute isolatedness of rational singularities and rational triple points,
Fonc. Anal. Appl. 2 (1968), pp. 324-332.

\bibitem{Yam} M. Yamamoto. Classification of isolated algebraic singularities by their
Alexander polynomials. Topology {\bf 23} (3) (1984), 277--287.

\end{thebibliography}
\end{document}